\documentclass[11pt]{amsart}
\usepackage[letterpaper,margin=1.2in]{geometry}
\usepackage{amsbsy,amsfonts,amsmath,amssymb,amscd,amsthm,mathrsfs}
\usepackage{subfigure,enumitem,graphicx,epsfig,color}
\usepackage[colorlinks=true,linkcolor=blue,citecolor=green]{hyperref}
\normalsize

\newtheorem{theorem}{Theorem}[section]
\newtheorem{lemma}[theorem]{Lemma}

\newtheorem{definition}[theorem]{Definition}
\newtheorem{corollary}[theorem]{Corollary}
\newtheorem{remark}[theorem]{Remark}

\newtheorem*{theorem A}{Theorem A}
\newtheorem*{corollary B}{Corollary B}
\newtheorem*{corollary C}{Corollary C}
\newtheorem*{theorem B}{Theorem B}
\theoremstyle{definition}
\newtheorem{example}[theorem]{Example}

\begin{document}
\title[Various shadowing properties for time varying maps]{Various shadowing properties for time varying maps}
\author[Javad Nazarian Sarkooh]{Javad Nazarian Sarkooh$^{*}$}
\address{Department of Mathematics, Ferdowsi University of Mashhad, Mashhad, IRAN.}
\email{\textcolor[rgb]{0.00,0.00,0.84}{javad.nazariansarkooh@gmail.com}}
 \thanks{$^*$Corresponding author}
\begin{abstract}
This paper is concerned with the study of various notions of shadowing of dynamical systems induced by a sequence of maps, so-called time varying maps, on a metric space. We define and study the shadowing, h-shadowing, limit shadowing, s-limit shadowing and exponential limit shadowing properties of these dynamical systems. We show that h-shadowing, limit shadowing and s-limit shadowing properties are conjugacy invariant. Also, we investigate the relationships between these notions of shadowing for time varying maps and examine the role that expansivity plays in shadowing properties of such dynamical systems. Specially, we prove some results linking s-limit shadowing property to limit shadowing property, and h-shadowing property to s-limit shadowing and limit shadowing properties. Moreover, under the assumption of expansivity, we show that the shadowing property implies the h-shadowing, s-limit shadowing and limit shadowing properties. Finally, it is proved that the uniformly contracting and uniformly expanding time varying maps exhibit the shadowing, limit shadowing, s-limit shadowing and exponential limit shadowing properties.
\end{abstract}

\maketitle
\thispagestyle{empty}

\section{Introduction}
The time varying maps (so-called non-autonomous or time-dependent dynamical systems), describe situations where the dynamics can vary with time and yield very flexible models than autonomous cases for the study and description of real world processes. They may be used to describe the evolution
of a wider class of phenomena, including systems which are forced or driven. For example, any moving picture
on a television screen is an example of time varying dynamical systems. 
In the recent past, lots of studies have been done regarding dynamical properties in such systems, but a global 
theory is still out of reach. Kolyada et al. \cite{KS,KST} gave definition of topological entropy of time varying maps and discussed minimality of these systems. Also, $\omega$-limit sets and attraction of time varying maps were studied in \cite{JSC,KRR,LLCB}. Then, stability of time varying maps were investigated \cite{BV, KWW}. Thakkar and Das \cite{DTRD} studied expansiveness, shadowing and topological stability of time varying maps.
In \cite{HXWXZF,JNSFHG1,K1,K2,K3,JNSFHG}, authors studied topological entropy, topological pressure and thermodynamic properties of time varying maps. Weak mixing and chaos of time varying maps
were also studied by \cite{JSC1,OPWP,SYCGG,TCCGG}. Ott et al. \cite{O} studied the evolution of probability distributions and exponential loss of memory for certain time-dependent dynamical systems.

In general time varying maps can be rather complicated. Thus, we are inclined to look at approximations of orbits, also called pseudo orbits. Systems for which pseudo orbits can be approximated by true orbits are said to satisfy the shadowing property. The shadowing property plays a key role in the study of the stability of dynamical systems.
This property is found in hyperbolic dynamics, and it was used to prove their stability, see for example \cite{MLLL}.
In this literature, some remarkable results were further obtained through works of several authors,
see e.g. \cite{ASAMMR,BADGCOP,CBKDDD,FNMSAA,VGVG,KR,SKKK}.

Since the approximation by true orbits can be expressed in various ways, different notions of shadowing have been introduced. In this paper, what we want to study on time varying maps is shadowing, h-shadowing, limit shadowing, s-limit shadowing and exponential limit shadowing properties.

\textbf{This is how the paper is organized:} 
In Section \ref{section2}, we give a precise definition of a time varying map, review the main concepts and set up our notation. In this section, shadowing, h-shadowing, limit shadowing, s-limit shadowing and exponential limit shadowing properties for time varying maps are considered. Then, we study the basic properties of these notions of shadowing in Section \ref{section3}. Especially, we show that the h-shadowing, limit shadowing and s-limit shadowing properties are conjugacy invariant. Also, by considering these notions of shadowing, we extend earlier results from other papers and identify some subtle changes to the theory in this case. In section \ref{section333}, we investigate the relationships between these notions of shadowing for time varying maps and examine the role that expansivity plays in shadowing properties of such dynamical systems. Specially, we prove some results linking s-limit shadowing property to limit shadowing property, and h-shadowing property to s-limit shadowing and limit shadowing properties. Moreover, under the assumption of expansivity, we show that the shadowing property implies the h-shadowing, s-limit shadowing and limit shadowing properties. Finally, in Section \ref{section4}, we prove that the uniformly expanding and uniformly contracting time varying maps exhibit the shadowing, limit shadowing, s-limit shadowing and exponential limit shadowing properties. Also, we show that any time varying map of a finite set of hyperbolic linear homeomorphisms on a Banach space with the same stable and unstable subspaces has the shadowing, limit shadowing, s-limit shadowing and exponential limit shadowing properties.
\section{Preliminaries}\label{section2}
Throughout this paper we consider $(X,d)$ to be a metric space, $f_{n}:X\to X$, $n\in\mathbb{N}$, to be a sequence of continuous maps and $\mathcal{F}=\{f_{n}\}_{n\in\mathbb{N}}$ to be a time varying map
on $X$ that its time evolution is defined by composing the maps $f_{n}$ in the following way
\begin{equation}
\mathcal{F}_{n}:=f_{n}\circ f_{n-1}\circ\cdots\circ f_{1},\ \textnormal{for}\ n\geq 1,\ \textnormal{and}\ \mathcal{F}_{0}:=Id_{X}.
\end{equation}

For time varying map $\mathcal{F}=\{f_{n}\}_{n\in\mathbb{N}}$ defined on $X$, we set $\mathcal{F}_{[i,j]}:=f_{j}\circ f_{j-1}\circ\cdots\circ f_{i+1}\circ f_{i}$ for $1\leq i\leq j$,
and $\mathcal{F}_{[i,j]}:=Id_{X}$ for $i>j$. Also, for any $k>0$, we define a
time varying map ($k^{th}$-iterate of $\mathcal{F}$) $\mathcal{F}^{k}=\{g_{n}\}_{n\in\mathbb{N}}$ on $X$, where
\begin{equation}
g_{n}=f_{nk}\circ f_{(n-1)k+k-1}\circ\ldots\circ f_{(n-1)k+2}\circ f_{(n-1)k+1}\ \text{for}\ n\geq 1.
\end{equation}
Thus $\mathcal{F}^{k}=\{\mathcal{F}_{[(n-1)k+1,nk]}\}_{n\in\mathbb{N}}$. Moreover, if time varying map $\mathcal{F}$ shifted $k$-times ($k\geq 1$), then we denote it by $\mathcal{F}(k,\textnormal{shift})$, i.e. $\mathcal{F}(k,\textnormal{shift})=\{f_{n}\}_{n=k+1}^{\infty}$.

Let $\mathcal{F}=\{f_{n}\}_{n\in\mathbb{N}}$ be a time varying map on a metric space $(X,d)$. For a point $x_{0}\in X$, put $x_{n}:=\mathcal{F}_{n}(x_{0})$ for all $n\geq 0$. Then the sequence $\{x_{n}\}_{n\geq 0}$, denoted by $\mathcal{O}(x_{0})$, is said to be the \emph{orbit} of $x_{0}$ under time varying map $\mathcal{F}=\{f_{n}\}_{n\in\mathbb{N}}$. Moreover, a subset $Y$ of $X$ is said to be \emph{invariant} under $\mathcal{F}$ if $f_{n}(Y)=Y$ for all $n\geq 1$, equivalently $\mathcal{F}_{n}(Y)=Y$ for all $n\geq 0$.
\begin{definition}[Conjugacy]
Let $(X,d_{1})$ and $(Y,d_{2})$ be two metric spaces. Let $\mathcal{F}=\{f_{n}\}_{n\in\mathbb{N}}$ and $\mathcal{G}=\{g_{n}\}_{n\in\mathbb{N}}$ be time varying maps on $X$ and $Y$, respectively. If there exists a homeomorphism $h:X\to Y$ such that $h\circ f_{n}=g_{n}\circ h$, for all $n\in\mathbb{N}$, then $\mathcal{F}$ and $\mathcal{G}$ are said to be \emph{conjugate} (with respect to the map $h$) or $h$-\emph{conjugate}. In particular, if $h:X\to Y$ is a uniform homeomorphism, then $\mathcal{F}$ and $\mathcal{G}$ are said to be \emph{uniformly conjugate} or \emph{uniformly} $h$-\emph{conjugate}. (Recall that homeomorphism $h:X\to Y$, such that $h$ and $h^{-1}$ are uniformly continuous, is called a uniform homeomorphism.)
\end{definition}
For example, if $\mathcal{F}=\{x^{n+1}\}_{n\in\mathbb{N}}$ on $[0,1]$ and $\mathcal{G}=\{2((x+1)/2)^{n+1}-1\}_{n\in\mathbb{N}}$ on $[-1,1]$, then $\mathcal{F}$ is uniformly $h$-conjugate to $\mathcal{G}$, where $h:[0,1]\to [-1,1]$ is defned by $h(x)=2x-1$, see \cite{DTRD}.
\begin{definition}[Shadowing property]
Let $\mathcal{F}=\{f_{n}\}_{n\in\mathbb{N}}$ be a time varying map on a metric space $(X,d)$ and $Y$ be a subset of $X$. Then,
\begin{enumerate}
\item
for $\delta>0$, a sequence $\{x_{n}\}_{n\geq 0}$ in $X$ is said to be a $\delta$-\emph{pseudo orbit} if
\begin{equation*}
d(f_{n+1}(x_{n}),x_{n+1})<\delta\ \textnormal{for all}\ n\geq 0;
\end{equation*}
\item
for given $\varepsilon>0$, a $\delta$-pseudo orbit $\{x_{n}\}_{n\geq 0}$ is said to be
$\varepsilon$-\emph{shadowed} by $x\in X$ if $d(\mathcal{F}_{n}(x),x_{n})<\varepsilon$ for all $n\geq 0$;
\item
the time varying map $\mathcal{F}$ is said to have \emph{shadowing property} on $Y$ if, for every $\varepsilon>0$, there exists a $\delta>0$ such that every $\delta$-pseudo orbit in $Y$ is $\varepsilon$-shadowed by
some point of $X$. If this property holds on $Y=X$, we simply say that $\mathcal{F}$ has \emph{shadowing property}.
\end{enumerate}
\end{definition}
\begin{remark}
If $\mathcal{F}=\{f_{n}\}_{n\in\mathbb{N}}$ be a time varying map on a compact metric space $(X,d)$ and  $Y$ be a subset of $X$, then it is easy to see that the time varying map $\mathcal{F}$ has shadowing property on $Y$ if and only if for every $\varepsilon>0$ there is a $\delta>0$ such that every finite $\delta$-pseudo orbit in $Y$ is $\varepsilon$-shadowed by some point of $X$.
\end{remark}
\begin{definition}[h-shadowing property]
Let $\mathcal{F}=\{f_{n}\}_{n\in\mathbb{N}}$ be a time varying map on a metric space $(X,d)$ and $Y$ be a subset of $X$. We say that $\mathcal{F}$ has \emph{h-shadowing property} on $Y$ if 
for every $\varepsilon>0$ there exists $\delta>0$ such that for every finite $\delta$-pseudo orbit
$\{x_{0},x_{1},\ldots,x_{m}\}$ in $Y$ there is $x\in X$ such that $d(\mathcal{F}_{n}(x),x_{n})<\varepsilon$ for every $0\leq n<m$ and $\mathcal{F}_{m}(x)=x_{m}$.  If this property holds on $Y=X$, we simply say that $\mathcal{F}$ has \emph{h-shadowing property}.
\end{definition}

\begin{remark}\label{remark1}
It is easy to see that every time varying map on a compact metric space with h-shadowing property has shadowing property, but the converse is not true, see \cite[Example 6.4]{BADGCOP}. Note that each continuous map generates a time varying map.
\end{remark}
Now, we define the limit shadowing and s-limit shadowing properties on time varying maps.
\begin{definition}[Limit shadowing property]
Let $\mathcal{F}=\{f_{n}\}_{n\in\mathbb{N}}$ be a time varying map on a metric space $(X,d)$ and $Y$ be a subset of $X$. Then,
\begin{enumerate}
\item a sequence $\{x_{n}\}_{n\geq 0}$ in $X$ is called a \emph{limit pseudo orbit} if $d(f_{n+1}(x_{n}),x_{n+1})\to 0$ as $ n\to +\infty$;
\item a sequence $\{x_{n}\}_{n\geq 0}$ in $X$ is said to be \emph{limit shadowed} if there is $x\in X$ such that $d(\mathcal{F}_{n}(x),x_{n})\to 0$, as $n\to +\infty$;
\item the time varying map $\mathcal{F}$ has the \emph{limit shadowing property} on $Y$ whenever every limit pseudo orbit in $Y$ is limit shadowed by some point of $X$. If this property holds on $Y=X$, we simply say that $\mathcal{F}$ has \emph{limit shadowing property}.
\end{enumerate}
\end{definition}
The notion of limit shadowing property was extended to a notion so called s-limit shadowing property, to account the fact that many systems have limit shadowing property but not shadowing property. 
\begin{definition}[s-Limit shadowing property]
Let $\mathcal{F}=\{f_{n}\}_{n\in\mathbb{N}}$ be a time varying map on a metric space $(X,d)$ and $Y$ be a subset of $X$. We say that $\mathcal{F}$ has \emph{s-limit shadowing property} on $Y$ if for
every $\varepsilon > 0$ there is $\delta > 0$ such that
\begin{enumerate}
\item for every $\delta$-pseudo orbit $\{x_{n}\}_{n\geq 0}$ in $Y$, there exists $x\in X$ satisfying $d(\mathcal{F}_{n}(x),x_{n})<\varepsilon$ for all $n\geq 0$, and,
\item if in addition, $\{x_{n}\}_{n\geq 0}$ is a limit pseudo orbit in $Y$ then $d(\mathcal{F}_{n}(x),x_{n})\to 0$ as $n\to +\infty$.
\end{enumerate}
If this property holds on $Y=X$, we simply say that $\mathcal{F}$ has \emph{s-limit shadowing property}.
\end{definition}
\begin{example}\label{example1}
Let $X=[0,1]\cup\{-1/2^{n}:n\geq 1\}$ and $f:X\to X$ be any
homeomorphism such that $f(x)=x$ for $x=1$ or $x\leq 0$ and $f(x)<x$ for
$x\in(0,1)$. Put $f_{n}=f$ for every $n\in\mathbb{N}$. Then time varying map $\mathcal{F}=\{f_{n}\}_{n\in\mathbb{N}}$ on $X$ has shadowing and limit shadowing properties, but it
does not have s-limit shadowing property, see \cite[Example 3.5]{BADGCOP}.
\end{example}
We say that a sequence $\{a_{n}\}_{n\geq 0}$ of real numbers converges to zero with rate $\theta\in(0,1)$ and write $a_{n}\xrightarrow{\theta} 0$ as $ n\to +\infty$, if there exists a constant $L>0$ such that $|a_{n}|\leq L\theta^{n}$ for all $n\geq 0$.
\begin{definition}[Exponential limit shadowing property]
Let $\mathcal{F}=\{f_{n}\}_{n\in\mathbb{N}}$ be a time varying map on a metric space $(X,d)$ and $Y$ be a subset of $X$. Then,
\begin{enumerate}
\item
for $\theta\in(0,1)$, a sequence $\{x_{n}\}_{n\geq 0}$ in $X$ is called a $\theta$-\emph{exponentially
limit pseudo orbit} of $\mathcal{F}$ if $d(f_{n+1}(x_{n}),x_{n+1})\xrightarrow{\theta} 0$ as $ n\to +\infty$;
\item
the time varying map $\mathcal{F}$ has the \emph{exponential limit shadowing property with exponent} $\xi$ on $Y$ if there exists $\theta_{0}\in(0,1)$ so that for any $\theta$-exponentially limit pseudo orbit $\{x_{n}\}_{n\geq 0}\subseteq Y$ with $\theta\in(\theta_{0},1)$, there is $x\in X$ such that $d(\mathcal{F}_{n}(x),x_{n})\xrightarrow{\theta^{\xi}} 0$, as $n\to +\infty$. In the case $\xi=1$ we say that F has the \emph{exponential limit shadowing property} on $Y$. If this property holds on $Y=X$, we simply say that $\mathcal{F}$ has \emph{exponential limit shadowing property}.
\end{enumerate}
\end{definition}

\section{Basic properties of various notions of shadowing}\label{section3}
Our aim of this section is to characterize the basic properties of various notions of shadowing (i.e. shadowing, h-shadowing, limit shadowing, s-limit shadowing and exponential limit shadowing properties) of time varying maps. Specially, we show that h-shadowing, limit shadowing and s-limit shadowing properties are conjugacy invariant. Moreover, by considering these notions of shadowing, we are able to extend earlier results from other papers, identifying some subtle changes to the theory in this case.

Thakkar and Das in \cite[Theorem 3.1]{DTRD} show that the shadowing property of time varying maps is conjugacy invariant. Now, in the following theorem, we show that the h-shadowing, limit shadowing and s-limit shadowing properties are conjugacy invariant. Our approach is similar to \cite[Theorem 3.1]{DTRD} and for completeness we give its proof.
\begin{theorem}
Let $\mathcal{F}=\{f_{n}\}_{n\in\mathbb{N}}$ and $\mathcal{G}=\{g_{n}\}_{n\in\mathbb{N}}$ be time varying maps on metric spaces $(X,d_{1})$ and $(Y,d_{2})$, respectively, such that $\mathcal{F}$ is uniformly conjugate to $\mathcal{G}$. Then, the following statements hold:
\begin{itemize}
\item[\textnormal{(a)}]
If $\mathcal{F}$ has the h-shadowing property, then so does $\mathcal{G}$.
\item[\textnormal{(b)}]
If $\mathcal{F}$ has the limit shadowing property, then so does $\mathcal{G}$.
\item[\textnormal{(c)}]
If $\mathcal{F}$ has the s-limit shadowing property, then so does $\mathcal{G}$.
\end{itemize}
\end{theorem}
\begin{proof}
Since $\mathcal{F}$ is uniformly conjugate to $\mathcal{G}$, there exists a uniform homeomorphism
$h:X\to Y$ such that $h\circ f_{n}=g_{n}\circ h$ for all $n\in\mathbb{N}$, which implies
$f_{n}\circ h^{-1}=h^{-1}\circ g_{n}$ for all $n\in\mathbb{N}$. Hence, for all $n\geq 0$,
\begin{eqnarray*}
h\circ\mathcal{F}_{n}
&=& h\circ f_{n}\circ f_{n-1}\circ\cdots\circ f_{1}\\
&=& g_{n}\circ h\circ f_{n-1}\circ\cdots\circ f_{1}\\
&\vdots &\\
&=& g_{n}\circ g_{n-1}\circ\cdots\circ g_{1}\circ h\\
&=& \mathcal{G}_{n}\circ h.
\end{eqnarray*}

\textnormal{(a)}. Let $\varepsilon>0$ be given. By uniform continuity of $h$ there exists an $\varepsilon_{0}>0$ such that $d_{1}(x,y)<\varepsilon_{0}$ implies $d_{2}(h(x),h(y))<\varepsilon$. Since $\mathcal{F}$ has h-shadowing property there exists a $\delta_{0}>0$ such that any finite $\delta_{0}$-pseudo orbit of $\mathcal{F}$ is $\varepsilon_{0}$-shadowed (with exact hit at the end) by $\mathcal{F}$ orbit of some point of $X$. Since $h^{-1}$ is uniformly continuous, for $\delta_{0}>0$ there exists a  $\delta>0$ such that $d_{2}(x,y)<\delta$ implies $d_{1}(h^{-1}(x),h^{-1}(y))<\delta_{0}$. Now, let $\{x_{0},x_{1},\ldots,x_{m}\}$ be a finite $\delta$-pseudo orbit
for $\mathcal{G}$, i.e. $d_{2}(g_{n+1}(x_{n}),x_{n+1})<\delta$ for every $0\leq n<m$. Hence 
$d_{1}(h^{-1}(g_{n+1}(x_{n})),h^{-1}(x_{n+1}))<\delta_{0}$, and so
$d_{1}(f_{n+1}(h^{-1}(x_{n})),h^{-1}(x_{n+1}))<\delta_{0}$. Therefore
$\{h^{-1}(x_{0}),h^{-1}(x_{1}),\ldots,h^{-1}(x_{m})\}$ is a finite $\delta_{0}$-pseudo orbit
for $\mathcal{F}$. Thus there exists a $x\in X$ such that $d_{1}(\mathcal{F}_{n}(x),h^{-1}(x_{n}))<\varepsilon_{0}$ for every $0\leq n<m$ and $\mathcal{F}_{m}(x)=h^{-1}(x_{m})$. Hence, $d_{2}(h(\mathcal{F}_{n}(x)),x_{n})<\varepsilon$ for every $0\leq n<m$ and $h(\mathcal{F}_{m}(x))=x_{m}$. Thus $d_{2}(\mathcal{G}_{n}(h(x)),x_{n})<\varepsilon$ for every $0\leq n<m$ and $\mathcal{G}_{m}(h(x))=x_{m}$. Thus $\mathcal{G}$ also has the h-shadowing property.

\textnormal{(b)}. Let $\mathcal{F}$ has the limit shadowing property, and let $\{y_{n}\}_{n\geq 0}$ be a limit pseudo orbit of $\mathcal{G}$, i.e. $d_{2}(g_{n+1}(y_{n}),y_{n+1})\to 0$ as $n\to +\infty$. Then, by uniform continuity of $h^{-1}$,
$d_{1}(h^{-1}\circ g_{n+1}(y_{n}),h^{-1}(y_{n+1}))\to 0$ as $n\to +\infty$, and so $d_{1}(f_{n+1}\circ h^{-1}(y_{n}),h^{-1}(y_{n+1}))\to 0$ as $n\to +\infty$. Thus $\{h^{-1}(y_{n})\}_{n\geq 0}$ is a limit pseudo orbit of $\mathcal{F}$. Hence, there exists $x\in X$ such that $d_{1}(\mathcal{F}_{n}(x),h^{-1}(y_{n}))\to 0$ as $n\to +\infty$. Again, by uniform continuity of $h$, $d_{2}(h\circ\mathcal{F}_{n}(x),h\circ h^{-1}(y_{n}))\to 0$ as $n\to +\infty$, and so $d_{2}(\mathcal{G}_{n}\circ h(x),h\circ h^{-1}(y_{n}))\to 0$ as $n\to +\infty$.
Hence, $d_{2}(\mathcal{G}_{n}(h(x)),y_{n})\to 0$ as $n\to +\infty$, which implies $\{y_{n}\}_{n\geq 0}$ is limit shadowed by $h(x)\in Y$. Consequently, any limit pseudo orbit of $\mathcal{G}$ can be limit shadowed by some point of $Y$. Thus $\mathcal{G}$ also has the limit shadowing property.

Finally, the part (c) is a direct consequence of part (b) and \cite[Theorem 3.1]{DTRD}.
\end{proof}
Thakkar and Das in \cite[Theorem 3.3]{DTRD} show that every finite direct product of time varying maps has the shadowing property if and only if all of its time varying maps have shadowing property. Now, in the following theorem, we extend this property to other notions of shadowing.
\begin{theorem}\label{theorem1}
Let $\mathcal{F}=\{f_{n}\}_{n\in\mathbb{N}}$ and $\mathcal{G}=\{g_{n}\}_{n\in\mathbb{N}}$ be time varying maps on metric spaces $(X,d_{1})$ and $(Y,d_{2})$, respectively. Define metric $d$ on $X\times Y$ by
\begin{equation*}
d((x_{1},y_{1}),(x_{2},y_{2})):=\max\{d_{1}(x_{1},x_{2}),d_{2}(y_{1},y_{2})\}\quad \text{for \ any} \ (x_{1},y_{1}),(x_{2},y_{2})\in X\times Y.
\end{equation*}
Then,
\begin{itemize}
\item[\textnormal{(a)}]
$\mathcal{F}$ and $\mathcal{G}$ have the h-shadowing property if and only if so does $\mathcal{F}\times\mathcal{G}:=\{f_{n}\times g_{n}\}_{n\in\mathbb{N}}$.
\item[\textnormal{(b)}]
$\mathcal{F}$ and $\mathcal{G}$ have the limit shadowing property if and only if so does $\mathcal{F}\times\mathcal{G}$.
\item[\textnormal{(c)}]
$\mathcal{F}$ and $\mathcal{G}$ have the exponential limit shadowing property if and only if so does $\mathcal{F}\times\mathcal{G}$.
\item[\textnormal{(d)}]
$\mathcal{F}$ and $\mathcal{G}$ have the s-limit shadowing property if and only if so does $\mathcal{F}\times\mathcal{G}$.
\end{itemize}
\end{theorem}
\begin{proof}
By definitions and \cite[Theorem 3.3]{DTRD} the proof of parts (a), (b) and (d) are not difficult, hence we only prove the part (c).

\textnormal{(c)}. Let time varying maps $\mathcal{F}$ and $\mathcal{G}$ have the exponential limit shadowing property. Then, there exists $\theta_{\mathcal{F}}\in(0,1)$ such that for any $\theta$-exponentially limit pseudo orbit $\{x_{n}\}_{n\geq 0}$ of $\mathcal{F}$ with $\theta\in(\theta_{\mathcal{F}},1)$, there is $x\in X$ such that  $d(\mathcal{F}_{n}(x),x_{n})\xrightarrow{\theta} 0$,
as $n\to +\infty$. Also, there exists $\theta_{\mathcal{G}}\in(0,1)$ such that for any $\theta$-exponentially limit pseudo orbit $\{y_{n}\}_{n\geq 0}$ of $\mathcal{G}$ with $\theta\in(\theta_{\mathcal{G}},1)$, there is $y\in Y$
such that $d(\mathcal{G}_{n}(y),y_{n})\xrightarrow{\theta} 0$, as $n\to +\infty$. Put $\theta_{0}=\max\{\theta_{\mathcal{G}},\theta_{\mathcal{F}}\}$, and let $\{(x_{n},y_{n})\}_{n\geq 0}$ be a $\theta$-exponentially limit pseudo orbit of $\mathcal{F}\times\mathcal{G}$ with $\theta\in(\theta_{0},1)$, i.e. $d((f_{n+1}\times g_{n+1})(x_{n},y_{n}),(x_{n+1},y_{n+1}))\xrightarrow{\theta} 0$ as $ n\to +\infty$.
Then, $d_{1}(f_{n+1}(x_{n}),x_{n+1})\xrightarrow{\theta} 0$ and $d_{2}(g_{n+1}(y_{n}),y_{n+1})\xrightarrow{\theta} 0$, as $ n\to +\infty$. Hence, there exist $x\in X$ and $y\in Y$ such that $d_{1}(\mathcal{F}_{n}(x),x_{n})\xrightarrow{\theta} 0$ and $d_{2}(\mathcal{G}_{n}(y),y_{n})\xrightarrow{\theta} 0$, as $n\to +\infty$. Therefore,
\begin{eqnarray*}
d((\mathcal{F}\times\mathcal{G})_{n}(x,y),(x_{n},y_{n}))
&=& d((\mathcal{F}_{n}(x),\mathcal{G}_{n}(y)),(x_{n},y_{n}))\\
&\leq& \max\{d_{1}(\mathcal{F}_{n}(x),x_{n}),d_{2}(\mathcal{G}_{n}(y),y_{n})\}\xrightarrow{\theta} 0\ \textnormal{as}\ n\to +\infty
\end{eqnarray*}
which implies $\mathcal{F}\times\mathcal{G}$ also has the exponential limit shadowing property.

Conversely, let the direct product $\mathcal{F}\times\mathcal{G}$ has the exponential limit shadowing property, and so there exists $\theta_{0}\in(0,1)$ such that for any $\theta$-exponentially limit pseudo orbit $\{(x_{n},y_{n})\}_{n\geq 0}$ of $\mathcal{F}\times\mathcal{G}$ with $\theta\in(\theta_{0},1)$, there is $(x,y)\in X\times Y$ such that  $d((\mathcal{F}\times\mathcal{G})_{n}(x,y),(x_{n},y_{n}))\xrightarrow{\theta} 0$,
as $n\to +\infty$. Now, let $\{x_{n}\}_{n\geq 0}$ and $\{y_{n}\}_{n\geq 0}$ be $\theta$-exponentially limit pseudo orbits of $\mathcal{F}$ and $\mathcal{G}$ with $\theta\in(\theta_{0},1)$, respectively, i.e. $d_{1}(f_{n+1}(x_{n}),x_{n+1})\xrightarrow{\theta} 0$ and $d_{2}(g_{n+1}(y_{n}),y_{n+1})\xrightarrow{\theta} 0$, as $ n\to +\infty$. Hence,
\begin{eqnarray*}
d((f_{n+1}\times g_{n+1})(x_{n},y_{n}),(x_{n+1},y_{n+1}))
&=& d((f_{n+1}(x_{n}), g_{n+1}(y_{n})),(x_{n+1},y_{n+1}))\\
&=& \max\{d_{1}(f_{n+1}(x_{n}),x_{n+1}),d_{2}(g_{n+1}(y_{n}),y_{n+1})\}\\
&\leq& d_{1}(f_{n+1}(x_{n}),x_{n+1})+d_{2}(g_{n+1}(y_{n}),y_{n+1})\xrightarrow{\theta} 0
\end{eqnarray*}
as $ n\to +\infty$. Therefore, $\{(x_{n},y_{n})\}_{n\geq 0}$ is a $\theta$-exponentially limit pseudo orbit of $\mathcal{F}\times\mathcal{G}$, and so there exists $(x,y)\in X\times Y$ such that
$d((\mathcal{F}\times\mathcal{G})_{n}(x,y),(x_{n},y_{n}))=d((\mathcal{F}_{n}(x),\mathcal{G}_{n}(y)),(x_{n},y_{n}))\xrightarrow{\theta} 0$ as $n\to +\infty$. Hence, $d_{1}(\mathcal{F}_{n}(x),x_{n})\xrightarrow{\theta} 0$ and $d_{2}(\mathcal{G}_{n}(y),y_{n})\xrightarrow{\theta} 0$, as $n\to +\infty$, which implies $\mathcal{F}$ and $\mathcal{G}$ have also the exponential limit shadowing property.
\end{proof}
Hence, every finite direct product of time varying maps has the h-shadowing, limit shadowing, s-limit shadowing
or exponential limit shadowing property if and only if all of its time varying maps have h-shadowing, limit shadowing, s-limit shadowing or exponential limit shadowing property, respectively.
\begin{theorem}
Let $\mathcal{F}=\{f_{n}\}_{n\in\mathbb{N}}$ be a time varying map on metric space $(X,d)$ and $k\in\mathbb{N}$.
Then, the following statements hold:
\begin{itemize}
\item[\textnormal{(a)}]
If $\mathcal{F}$ has the limit shadowing property, then so does $\mathcal{F}^{k}$.
\item[\textnormal{(b)}]
If $\mathcal{F}$ has the exponential limit shadowing property, then so does $\mathcal{F}^{k}$.
\item[\textnormal{(c)}]
If $\mathcal{F}$ has the s-limit shadowing property, then so does $\mathcal{F}^{k}$.
\end{itemize}
\end{theorem}
\begin{proof}
\textnormal{(a)}. The case $k=1$ is trivial, so let $k\geq 2$ and $\{y_{n}\}_{n\geq0}$ be a limit pseudo orbit of $\mathcal{F}^{k}$. Then $d(g_{n+1}(y_{n}),y_{n+1})\to 0$ as $ n\to +\infty$, where $g_{n}=\mathcal{F}_{[(n-1)k+1,nk]}$ for all $n\in\mathbb{N}$, and so $d(\mathcal{F}_{[nk+1,(n+1)k]}(y_{n}),y_{n+1})\to 0$ as $ n\to +\infty$. Put
\begin{equation}\label{jjj13}
x_{nk+j}:=\mathcal{F}_{[nk+1,nk+j]}(y_{n})\ \textnormal{for}\ 0\leq j<k\ \textnormal{and}\ n\geq 0.
\end{equation}
$\mathbf{Claim.}$ The sequence $\{x_{n}\}_{n\geq0}$ is a limit pseudo orbit for $\mathcal{F}$, i.e. $d(f_{nk+j+1}(x_{nk+j}),x_{nk+j+1})\to 0$ as $ n\to +\infty$, for all $n\geq 0$ and $0\leq j<k$.

To prove the claim, choose any $n\geq 0$. Then for any $0\leq j<k-1$,
\begin{equation*}
f_{nk+j+1}(x_{nk+j})=f_{nk+j+1}(\mathcal{F}_{[nk+1,nk+j]}(y_{n}))=\mathcal{F}_{[nk+1,nk+j+1]}(y_{n})=x_{nk+j+1}.
\end{equation*}
Thus $d(f_{nk+j+1}(x_{nk+j}),x_{nk+j+1})=0$, for all $n\geq 0$ and $0\leq j<k-1$. Now for $j=k-1$,
\begin{eqnarray*}
d(f_{nk+k}(x_{nk+k-1}),x_{nk+k})
&=& d(f_{nk+k}(\mathcal{F}_{[nk+1,nk+k-1]}(y_{n})),x_{nk+k})\\
&=& d(\mathcal{F}_{[nk+1,(n+1)K]}(y_{n}),y_{n+1})\\
&=& d(g_{n+1}(y_{n}),y_{n+1}).
\end{eqnarray*}
Hence, $d(f_{n+1}(x_{n}),x_{n+1})\to 0$ as $ n\to +\infty$, since $d(g_{n+1}(y_{n}),y_{n+1})\to 0$ as $ n\to +\infty$, which completes the proof of the claim.

Now, by the limit shadowing property of $\mathcal{F}$, $\{x_{n}\}_{n\geq0}$ is limit shadowed by some $x\in X$, i.e.
$d(\mathcal{F}_{n}(x),x_{n})\to 0$ as $n\to +\infty$. In particular, $d(\mathcal{F}_{kn}(x),x_{kn})\to 0$ as $n\to +\infty$. Thus $d(\mathcal{F}^{k}_{n}(x),y_{n})\to 0$ as $n\to +\infty$, since $y_{n}=x_{kn}$ and $\mathcal{F}^{k}_{n}=\mathcal{F}_{kn}$, which implies $\{y_{n}\}_{n\geq 0}$ is limit shadowed by $x\in X$. Consequently, any limit pseudo orbit of $\mathcal{F}^{k}$ can be limit shadowed by some point of $X$. Thus $\mathcal{F}^{k}$ has also the limit shadowing property.

\textnormal{(b)}.
For $k=1$ this is trivial, so let $k\geq 2$. Since the time varying map $\mathcal{F}$ has the exponential limit shadowing property, there exists $\theta_{0}\in(0,1)$ such that for any $\theta$-exponentially limit pseudo orbit $\{x_{n}\}_{n\geq 0}$ of $\mathcal{F}$ with $\theta\in(\theta_{0},1)$, there is $x\in X$ such that  $d(\mathcal{F}_{n}(x),x_{n})\xrightarrow{\theta} 0$, as $n\to +\infty$. Put $\theta_{1}:=\theta_{0}^{1/k}$. Now, let $\{y_{n}\}_{n\geq0}$ be a $\theta$-exponentially limit pseudo orbit of $\mathcal{F}^{k}$ with $\theta\in(\theta_{1},1)$, i.e. $d(g_{n+1}(y_{n}),y_{n+1})\xrightarrow{\theta} 0$ as $ n\to +\infty$, where $g_{n}=\mathcal{F}_{[(n-1)k+1,nk]}$ for all $n\in\mathbb{N}$. Hence, there is $L>0$ such that $d(g_{n+1}(y_{n}),y_{n+1})\leq L\theta^{n}$ for all $n\geq 0$. Consider the sequence $\{x_{n}\}_{n\geq 0}$ given by relation (\ref{jjj13}). Then, one has for all $n\geq 0$ and $0\leq j<k-1$:
\begin{equation*}
d(f_{nk+j+1}(x_{nk+j}),x_{nk+j+1})=0\leq (L\theta^{\frac{1-k}{k}})(\theta^{1/k})^{nk+j},
\end{equation*}
and for $j=k-1$:
\begin{equation*}
d(f_{nk+k}(x_{nk+k-1}),x_{nk+k})=d(g_{n+1}(y_{n}),y_{n+1})\leq L\theta^{n}=(L\theta^{\frac{1-k}{k}})(\theta^{1/k})^{nk+k-1}.
\end{equation*}
Thus $\{x_{n}\}_{n\geq0}$ is a $\theta^{1/k}$-exponentially limit pseudo orbit for $\mathcal{F}$. Hence, by exponential limit shadowing property of $\mathcal{F}$, there is $x\in X$ such that
$d(\mathcal{F}_{n}(x),x_{n})\xrightarrow{\theta^{1/k}} 0$ as $n\to +\infty$. In particular, $d(\mathcal{F}_{kn}(x),x_{kn})\xrightarrow{\theta^{1/k}} 0$ as $n\to +\infty$. Thus $d(\mathcal{F}^{k}_{n}(x),y_{n})\xrightarrow{\theta} 0$ as $n\to +\infty$, which implies $\mathcal{F}^{k}$ has the exponential limit shadowing property.

Finally, part $\textnormal{(c)}$ is a direct consequence of part $\textnormal{(a)}$ and \cite[Theorem 3.3]{DTRD}, which completes the proof of the theorem.
\end{proof}
In the following theorem, we show that if we can find some iterate of a time varying map which has h-shadowing property, then we can deduce that the time varying map itself has h-shadowing property. We need the assumption of equicontinuity.
\begin{definition}[Equicontinuity]
Time varying map $\mathcal{F}=\{f_{n}\}_{n\in\mathbb{N}}$ on a metric space $(X,d)$ is said to be \emph{equicontinuous} if for each $\varepsilon>0$ there exists $\delta>0$ such that $d(x,y)<\delta$ implies $d(\mathcal{F}_{[i,j]}(x),\mathcal{F}_{[i,j]}(y))<\varepsilon$ for all $1\leq i\leq j$.
\end{definition}
\begin{theorem}
Let $\mathcal{F}=\{f_{n}\}_{n\in\mathbb{N}}$ be an equicontinuous time varying map on a compact metric space $(X,d)$ and $Y$ be an invariant subset of $X$. Then, the following conditions are equivalent:
\begin{itemize}
\item[\textnormal{(a)}]
$\mathcal{F}$ has the h-shadowing property on $Y$.
\item[\textnormal{(b)}]
$\mathcal{F}^{k}$ has the h-shadowing property on $Y$ for some $k\in\mathbb{N}$.
\item[\textnormal{(c)}]
$\mathcal{F}^{k}$ has the h-shadowing property on $Y$ for all $k\in\mathbb{N}$.
\end{itemize}
\end{theorem}
\begin{proof}
First we prove $(a)\Rightarrow(c)$. By the h-shadowing property of $\mathcal{F}$, for $\varepsilon>0$ there exists $\delta>0$ such that every finite $\delta$-pseudo orbit in $Y$ is $\varepsilon$-shadowed (with exact hit at
the end) by $\mathcal{F}$ orbit of some point in $X$. Now, let $\{x_{0},x_{1},\ldots,x_{m}\}\subseteq Y$ be a finite $\delta$-pseudo orbit for $\mathcal{F}^{k}=\{g_{n}\}_{n\in\mathbb{N}}$ , where $g_{n}=\mathcal{F}_{[(n-1)k+1,nk]}$. Then the sequence
\begin{center}
$\{x_{0},\mathcal{F}_{[1,1]}(x_{0}),\mathcal{F}_{[1,2]}(x_{0}),\ldots,\mathcal{F}_{[1,k-1]}(x_{0}),x_{1},\mathcal{F}_{[k+1,k+1]}(x_{1}),\mathcal{F}_{[k+1,k+2]}(x_{1}),\ldots$
\end{center}
\begin{center}
$,\mathcal{F}_{[k+1,2k-1]}(x_{1}),x_{2},\ldots,x_{m-1},\mathcal{F}_{[(m-1)k+1,(m-1)k+1]}(x_{m-1}),\mathcal{F}_{[(m-1)k+1,(m-1)k+2]}(x_{m-1}),\ldots$
\end{center}
\begin{center}
$,\mathcal{F}_{[(m-1)k+1,mk-1]}(x_{m-1}),x_{m}\}$
\end{center}
is a finite $\delta$-pseudo orbit for $\mathcal{F}$ which $\varepsilon$-shadowed by some point $x\in X$ 
such that $\mathcal{F}_{mk}(x)=x_{m}$. Hence $\mathcal{F}^{k}$ has the h-shadowing property on $Y$ for all $k\in\mathbb{N}$ because of $g_{n}\circ\cdots\circ g_{1}=\mathcal{F}_{nk}$ for $1\leq n\leq m$.

Implication $(c)\Rightarrow(b)$ is trivial. 

To prove $(b)\Rightarrow(a)$, fix $\varepsilon>0$ and suppose that $\mathcal{F}^{k}$ has h-shadowing property on $Y$ for some $k\in\mathbb{N}$. Since time varying map $\mathcal{F}$ is equicontinuous and $X$ is compact, there exists $\eta>0$ such that $d(x,y)<\eta$ implies 
$d(\mathcal{F}_{[n,n+i]}(x),\mathcal{F}_{[n,n+i]}(y))<\frac{\varepsilon}{2}$ for every $n\geq 1$ and $0\leq i\leq k$.

By the h-shadowing property of $\mathcal{F}^{k}$ there exists $0<\delta<\frac{\varepsilon}{2}$ such that each finite $\delta$-pseudo orbit of $\mathcal{F}^{k}$ is $\eta$-shadowed by $\mathcal{F}^{k}$ orbit of some point of $X$ which hits the last element of the pseudo orbit. Since time varying map $\mathcal{F}$ is equicontinuous and $X$ is compact, there exists $0<\gamma<\frac{\delta}{k}$ such that $d(x,y)<\gamma$ implies $d(\mathcal{F}_{[n,n+i]}(x),\mathcal{F}_{[n,n+i]}(y))<\frac{\delta}{k}$ for every $n\geq 1$ and $0\leq i\leq k$. 

Now, let $\{x_{0},x_{1},\ldots,x_{m}\}\subseteq Y$ be any finite $\gamma$-pseudo orbit for $\mathcal{F}$ and write $m=sk+r$ for some $s\geq 0$ and some $0\leq r<k$. Then the sequence 
\begin{center}
$\{x_{0},x_{1},\ldots,x_{m},\mathcal{F}_{[m+1,m+1]}(x_{m}),\mathcal{F}_{[m+1,m+2]}(x_{m}),\ldots,\mathcal{F}_{[m+1,m+k-r]}(x_{m})\}\subseteq Y$
\end{center}
is a finite $\gamma$-pseudo orbit for $\mathcal{F}$ (note that $Y$ is an invariant subset of $X$), which we enumerate obtaining the sequence $\{x_{0},x_{1},\ldots,x_{(s+1)k}\}$. We claim that $\{x_{0},x_{k},x_{2k},\ldots,x_{(s+1)k}\}$
is a finite $\delta$-pseudo orbit for $\mathcal{F}^{k}$. Indeed,
\begin{eqnarray*}
d(\mathcal{F}_{[1,k]}(x_{0}),x_{k})
&\leq& d(x_{k},\mathcal{F}_{[k,k]}(x_{k-1}))+d(\mathcal{F}_{[k,k]}(x_{k-1}),\mathcal{F}_{[k-1,k]}(x_{k-2}))\\
&& +\cdots+d(\mathcal{F}_{[2,k]}(x_{1}),\mathcal{F}_{[1,k]}(x_{0}))\\
&<& \gamma+(k-1)\dfrac{\delta}{k}<\delta.
\end{eqnarray*}
Similarly for $1\leq i\leq s$, we have $d(\mathcal{F}_{[(i-1)k+1,ik]}(x_{(i-1)k}),x_{ik})<\delta$. Finally,
\begin{eqnarray*}
d(\mathcal{F}_{[sk+1,(s+1)k]}(x_{sk}),x_{(s+1)k})
&\leq& d(x_{(s+1)k},\mathcal{F}_{[(s+1)k,(s+1)k]}(x_{(s+1)k-1}))\\
&& +d(\mathcal{F}_{[(s+1)k,(s+1)k]}(x_{(s+1)k-1}),\mathcal{F}_{[(s+1)k-1,(s+1)k]}(x_{(s+1)k-2}))\\
&& +\cdots+d(\mathcal{F}_{[sk+2,(s+1)k]}(x_{sk+1}),\mathcal{F}_{[sk+1,(s+1)k]}(x_{sk}))\\
&<& r(\dfrac{\delta}{k})<\delta.
\end{eqnarray*}
By h-shadowing property of $\mathcal{F}^{k}$ there is $x\in X$ such that $d(\mathcal{F}_{[1,ik]}(x),x_{ik})<\eta$ for every $0\leq i\leq s$ and $\mathcal{F}_{[1,(s+1)k]}(x)=x_{(s+1)k}$. Now, for every $0\leq j<k$ we have 
\begin{center}
$d(\mathcal{F}_{[1,ik+j]}(x),\mathcal{F}_{[ik+1,ik+j]}(x_{ik}))<\dfrac{\varepsilon}{2}$
\end{center}
and
\begin{eqnarray*}
d(\mathcal{F}_{[ik+1,ik+j]}(x_{ik}),x_{ik+j})
&\leq& d(x_{ik+j},\mathcal{F}_{[ik+j,ik+j]}(x_{ik+j-1}))\\
&& +d(\mathcal{F}_{[ik+j,ik+j]}(x_{ik+j-1}),\mathcal{F}_{[ik+j-1,ik+j]}(x_{ik+j-2}))\\
&& +\cdots+d(\mathcal{F}_{[ik+2,ik+j]}(x_{ik+1}),\mathcal{F}_{[ik+1,ik+j]}(x_{ik}))\\
&<& \gamma+(k-1)\dfrac{\delta}{k}<\delta<\dfrac{\varepsilon}{2}.
\end{eqnarray*}
So $d(\mathcal{F}_{[1,ik+j]}(x),x_{ik+j})<\varepsilon$. Also, 
$f_{m+k-r}\circ\cdots\circ f_{1}(x)=x_{(s+1)k}=f_{m+k-r}\circ\cdots\circ f_{m+1}(x_{m})$ which implies $\mathcal{F}_{m}(x)=x_{m}$. Hence $\mathcal{F}$ has the h-shadowing property on $Y$.
\end{proof}
\begin{lemma}
Let $\mathcal{F}=\{f_{n}\}_{n\in\mathbb{N}}$ be a time varying map on a metric space $(X,d)$ and $k\in\mathbb{N}$ such that $f_{1},f_{2},\ldots,f_{k}$ are surjective. Then, $\mathcal{F}$ has the limit shadowing
property if and only if $\mathcal{F}(k,\textnormal{shift})=\{f_{n}\}_{n=k+1}^{\infty}$ has the limit shadowing property.
\end{lemma}
\begin{proof}
Let $\mathcal{F}(k,\textnormal{shift})$ has the limit shadowing property, and let $\{x_{n}\}_{n\geq 0}$ be a limit pseudo orbit of $\mathcal{F}$, i.e. $d(f_{n+1}(x_{n}),x_{n+1})\to 0$ as $n\to +\infty$. Then, the sequence $\{y_{n}\}_{n\geq 0}$ in which $y_{n}=x_{n+k}$ is a limit pseudo orbit of $\mathcal{F}(k,\textnormal{shift})$, i.e. $d(f_{n+k+1}(y_{n}),y_{n+1})\to 0$ as $ n\to +\infty$. Hence, there exists $y\in X$ such that $d(\mathcal{F}_{n}(k,\textnormal{shift})(y),y_{n})\to 0$ as $n\to +\infty$, where $\mathcal{F}_{n}(k,\textnormal{shift})=f_{n+k}\circ\cdots\circ f_{k+2}\circ f_{k+1}$. Now, consider a preimage $x$ of $y$ under $\mathcal{F}_{k}$, i.e. $\mathcal{F}_{k}(x)=y$ (note that $f_{1},f_{2},\ldots,f_{k}$ are surjective). Then $d(\mathcal{F}_{n}(x),x_{n})\to 0$ as $n\to +\infty$, which implies $\{x_{n}\}_{n\geq 0}$ is limit shadowed by $x\in X$. Consequently, any limit pseudo orbit of $\mathcal{F}$ can be limit shadowed by some point of $X$. Thus $\mathcal{F}$ also has the limit shadowing property.

Conversely, let $\mathcal{F}$ has the limit shadowing property, and let $\{x_{n}\}_{n\geq 0}$ be a limit pseudo orbit of $\mathcal{F}(k,\textnormal{shift})$, i.e. $d(f_{n+k+1}(x_{n}),x_{n+1})\to 0$ as $ n\to +\infty$. Then the sequence $\{y_{n}\}_{n\geq 0}$, where $y_{n+k}=x_{n}$ for $n\geq 0$ and $y_{1},y_{2},\ldots,y_{k-1}$ are arbitrary points of $X$, is a limit pseudo orbit of $\mathcal{F}$, i.e. $d(f_{n+1}(y_{n}),y_{n+1})\to 0$ as $ n\to +\infty$. Hence, there exists $y\in X$ such that $d(\mathcal{F}_{n}(y),y_{n})\to 0$ as $n\to +\infty$.
Now, put $x=f_{k}\circ\cdots\circ f_{2}\circ f_{1}(y)$. Then $d(\mathcal{F}_{n}(k,\textnormal{shift})(x),x_{n})\to 0$ as $n\to +\infty$, which implies $\{x_{n}\}_{n\geq 0}$ is limit shadowed by $x\in X$. Consequently, any limit pseudo orbit of $\mathcal{F}(k,\textnormal{shift})$ can be limit shadowed by some point of $X$. Thus $\mathcal{F}(k,\textnormal{shift})$ also has the limit shadowing property.
\end{proof}
\section{Shadowing properties and expansivity}\label{section333}
In this section, we investigate the relationships between various notions of shadowing for time varying maps and examine the role that expansivity plays in shadowing properties of such dynamical systems. We prove some results linking s-limit shadowing property to limit shadowing property, and h-shadowing property to s-limit shadowing and limit shadowing properties. Finally, under the assumption of expansivity, we show that the shadowing property implies the h-shadowing, s-limit shadowing and limit shadowing properties.
\begin{lemma}\label{lemma1}
Let $\mathcal{F}=\{f_{n}\}_{n\in\mathbb{N}}$ be a time varying map on a metric space $(X,d)$ and $Y$ be a subset of $X$. If $Y\subseteq f_{n}(Y)$ for every $n\in\mathbb{N}$ and $\mathcal{F}$ has s-limit shadowing property on $Y$ then $\mathcal{F}$ also has limit shadowing property on $Y$. In particular, if $\mathcal{F}$ is a time varying map of surjective maps and has s-limit shadowing property then $\mathcal{F}$ also has limit shadowing property.
\end{lemma}
\begin{proof}
Fix $\varepsilon>0$ and let $\delta>0$ be given by the s-limit shadowing property.
Let $\{x_{n}\}_{n\geq 0}\subseteq Y$ be a limit pseudo orbit of $\mathcal{F}$, i.e. $d(f_{n+1}(x_{n}),x_{n+1})\to 0$ as $ n\to +\infty$. Then for some $n_{0}\in\mathbb{N}$,
$d(f_{n+1}(x_{n}),x_{n+1})<\delta$ for every $n\geq n_{0}$. By assumption $Y\subseteq f_{n}(Y)$ for every $n\in\mathbb{N}$, there is $y_{0}\in Y$ such that $\{y_{0},\mathcal{F}_{1}(y_{0}),\ldots,\mathcal{F}_{n-1}(y_{0})\}\subseteq Y$ and $\mathcal{F}_{n}(y_{0})=x_{n_{0}}$. Hence, the sequence $\{y_{0},\mathcal{F}_{1}(y_{0}),\ldots,\mathcal{F}_{n-1}(y_{0}),x_{n_{0}},x_{n_{0}+1},\ldots\}$ is $\delta$-pseudo orbit and limit pseudo orbit. By s-limit shadowing property, $\{y_{0},\mathcal{F}_{1}(y_{0}),\ldots,\mathcal{F}_{n-1}(y_{0}),x_{n_{0}},x_{n_{0}+1},\ldots\}$ is $\varepsilon$-shadowed and limit shadowed by some point $y\in X$. Therefore $\{x_{n}\}_{n\geq 0}$ is limit shadowed by $y$ which implies $\mathcal{F}$ has the limit shadowing property on $Y$.
\end{proof}
\begin{theorem}\label{theorem22}
Let $\mathcal{F}=\{f_{n}\}_{n\in\mathbb{N}}$ be a time varying map on a compact metric space $(X,d)$ and $Y$ be a closed subset of $X$. Then, the following statements hold:
\begin{itemize}
\item[\textnormal{(a)}]
If there is an open set $U$ such that $Y\subseteq U$ and $\mathcal{F}$ has h-shadowing property on
$U$, then $\mathcal{F}$ has s-limit shadowing property on $Y$. If in addition, $Y\subseteq f_{n}(Y)$ for every $n\in\mathbb{N}$ then $\mathcal{F}$ has limit shadowing property on $Y$.
\item[\textnormal{(b)}]
If $Y$ is invariant and $\mathcal{F}|_{Y}$ has h-shadowing property then $\mathcal{F}|_{Y}$ has s-limit shadowing property and limit shadowing property.
\item[\textnormal{(c)}]
If $\mathcal{F}$ has h-shadowing property then $\mathcal{F}$ has s-limit shadowing property. If in addition, $\mathcal{F}$ is a time varying map of surjective maps then $\mathcal{F}$ has limit shadowing property.
\end{itemize}
\end{theorem}
\begin{proof}
\textnormal{(a)}. Since $X$ is compact, then by remark \ref{remark1} every time varying map with h-shadowing property has shadowing property. Hence the first half of the definition of s-limit shadowing property is satisfied trivially.

So fix $\varepsilon>0$ such that $B(Y,3\varepsilon)\subseteq U$ and denote $\varepsilon_{n}=2^{-n-2}\varepsilon$ for every $n\in\mathbb{N}\cup\{0\}$ (note that $B(Y,r)$ is the $r$-neighborhood of the set $Y$). By the definition of h-shadowing property there are $\{\delta_{n}>0\}_{n\in\mathbb{N}\cup\{0\}}$ such that every finite $\delta_{n}$-pseudo orbit in $U$ is $\varepsilon_{n}$-shadowed by some point of $X$ (with exact hit at the end). Fix any $\delta_{0}$-pseudo orbit $\{x_{n}\}_{n\geq 0}\subseteq Y$ such that $d(f_{n+1}(x_{n}),x_{n+1})\to 0$ as $ n\to +\infty$. There is an increasing sequence $\{k_{i}\}_{i\in\mathbb{N}\cup\{0\}}$ such that $\{x_{n}\}_{n\geq k_{i}}$ is a $\delta_{i}$-pseudo orbit for $\mathcal{F}(k_{i},\textnormal{shift})=\{f_{n}\}_{n=k_{i}+1}^{\infty}$ and obviously $k_{0}=0$. Note that if $w$ is a point such that $\mathcal{F}_{k_{i}}(w)=x_{k_{i}}$ then
the sequence
\begin{center}
$\{w,\mathcal{F}_{1}(w),\ldots,\mathcal{F}_{k_{i}}(w),x_{k_{i}+1},\ldots,x_{k_{i+1}}\}$
\end{center}
is a finite $\delta_{i}$-pseudo orbit. Let $z_{0}$ be a point which $\varepsilon_{0}$-shadows the finite $\delta_{0}$-pseudo orbit $\{x_{0},\ldots,x_{k_{1}}\}$ with exact hit at the end, i.e. $\mathcal{F}_{k_{1}}(z_{0})=x_{k_{1}}$. Note that $\mathcal{F}_{j}(z_{0})\in U$ for $0\leq j\leq k_{1}$.
 
For $i\in\mathbb{N}$, assume that $z_{i}$ is a point which $\varepsilon_{i}$-shadows the finite $\delta_{i}$-pseudo orbit
\begin{center}
$\{z_{i-1},\mathcal{F}_{1}(z_{i-1}),\ldots,\mathcal{F}_{k_{i}}(z_{i-1}),x_{k_{i}+1},\ldots,x_{k_{i+1}}\}\subseteq U$
\end{center}
with exact hit at the end. Then by h-shadowing property there is a point $z_{i+1}$ which $\varepsilon_{i+1}$-shadows the finite $\delta_{i+1}$-pseudo orbit
\begin{center}
$\{z_{i},\mathcal{F}_{1}(z_{i}),\ldots,\mathcal{F}_{k_{i+1}}(z_{i}),x_{k_{i+1}+1},\ldots,x_{k_{i+2}}\}\subseteq U$
\end{center}
with exact hit at the end. Thus we can produce a sequence $\{z_{i}\}_{i\geq 0}$ with the following
properties:
\begin{itemize}
\item[\textnormal{(1)}]
$d(\mathcal{F}_{j}(z_{i-1}),\mathcal{F}_{j}(z_{i}))<\varepsilon_{i}$ for $0\leq j\leq k_{i}$ and $i\geq 1$;
\item[\textnormal{(2)}]
$d(\mathcal{F}_{j}(z_{i}),x_{j})<\varepsilon_{i}$ for $k_{i}<j\leq k_{i+1}$ and $i\geq 0$;
\item[\textnormal{(3)}]
$\mathcal{F}_{k_{i+1}}(z_{i})=x_{k_{i+1}}$ for $i\geq 0$;
\item[\textnormal{(4)}]
$d(\mathcal{F}_{j}(z_{i}),Y)<\varepsilon$ for $j\leq k_{i+1}$.
\end{itemize}
Since $X$ is compact, there is an increasing sequence $\{s_{i}\}_{i\geq 1}$ such that the limit $z=\lim_{i\to\infty}z_{s_{i}}$ exists. Hence, for any $j,n\in\mathbb{N}$ there exist $i_{0}\geq 0$ and $m\geq i_{0}$ such that $k_{i_{0}}<j\leq k_{i_{0}+1}$
and $d(\mathcal{F}_{j}(z),\mathcal{F}_{j}(z_{s_{m}}))<\varepsilon_{n+1}$. So we get
\begin{eqnarray*}
d(\mathcal{F}_{j}(z),x_{j})
&\leq& d(\mathcal{F}_{j}(z),\mathcal{F}_{j}(z_{s_{m}}))
+d(\mathcal{F}_{j}(z_{i_{0}}),x_{j})
+\sum_{i=i_{0}}^{s_{m}-1}d(\mathcal{F}_{j}(z_{i}),\mathcal{F}_{j}(z_{i+1}))\\
&<& \varepsilon_{n+1}+\varepsilon_{i_{0}}+\sum_{i=i_{0}}^{s_{m}-1}\varepsilon_{i+1}
<2^{-n-3}\varepsilon+\sum_{i=i_{0}}^{\infty}2^{-i-2}\varepsilon\\
&=& \varepsilon(2^{-n-3}+2^{-i_{0}-1})<\varepsilon.
\end{eqnarray*}
Furthermore, for any $n$, let $j>k_{n+2}$. Then, there is $i_{1}\geq n+2$ such that $k_{i_{1}}<j\leq k_{i_{1}+1}$
and there is $m>i_{1}$ such that $d(\mathcal{F}_{j}(z),\mathcal{F}_{j}(z_{s_{m}}))<\varepsilon_{n+1}$. Hence, as
before we obtain
\begin{center}
$d(\mathcal{F}_{j}(z),x_{j})<\varepsilon(2^{-n-3}+2^{-i_{1}-1})\leq\varepsilon(2^{-n-3}+
2^{-n-3})=\varepsilon_{n}.$
\end{center}
This immediately implies that $\limsup_{j\to\infty}d(\mathcal{F}_{j}(z),x_{j})\leq\varepsilon_{n}$. Since $n$
was arbitrary, we have $\lim_{j\to\infty}d(\mathcal{F}_{j}(z),x_{j})=0$. This shows that $\mathcal{F}$ has
s-limit shadowing property on $Y$.

Finally, (b) and (c) follow directly from (a) and lemma \ref{lemma1} (since $U=Y$ is open in $Y$), which
completes the proof of the theorem. 
\end{proof}
\begin{definition}[Expansivity]
An time varying map $\mathcal{F}=\{f_{n}\}_{n\in\mathbb{N}}$ on a metric space $(X,d)$ is called \emph{strongly expansive} if there exists $\gamma>0$ (called expansivity constant) such that for any two distinct points $x,y\in X$ and every $N\in\mathbb{N}$, $d(\mathcal{F}_{[N,n]}(x),\mathcal{F}_{[N, n]}(y))>\gamma$ for some $n\geq N$. Equivalently, if for $x,y\in X$ and some $N\in\mathbb{N}$,
$d(\mathcal{F}_{[N,n]}(x),\mathcal{F}_{[N, n]}(y))\leq\gamma$ for all $n\geq N$, then $x=y$.
\end{definition}
\begin{corollary}
Let $\mathcal{F}=\{f_{n}\}_{n\in\mathbb{N}}$ be a time varying map on a compact metric space $(X,d)$. 
\begin{itemize}
\item[\textnormal{(a)}]
If $\mathcal{F}$ is strongly expansive then $\mathcal{F}$ has the shadowing property if and only if $\mathcal{F}$ has the h-shadowing property.
\item[\textnormal{(b)}]
If $\mathcal{F}$ is strongly expansive and has the shadowing property then $\mathcal{F}$ has the h-shadowing and s-limit shadowing properties. If in addition, $\mathcal{F}$ is a time varying map of surjective maps then $\mathcal{F}$ has the limit shadowing property.
\end{itemize}
\end{corollary}
\begin{proof}
If $\mathcal{F}$ has the h-shadowing property then $\mathcal{F}$ has the shadowing property (see Remark \ref{remark1}). So suppose that $\mathcal{F}$ has the shadowing property. Let $\varepsilon<\gamma$ and let $\delta>0$ be provided by shadowing property for $\varepsilon$, where $\gamma$ is the expansivity constant. Fix any finite $\delta$-pseudo orbit $\{x_{0},x_{1},\ldots,x_{m}\}$ and extend it to the infinite $\delta$-pseudo orbit
\begin{center}
$$\{x_{0},x_{1},\ldots,x_{m},\mathcal{F}_{[m+1,m+1]}(x_{m}),\mathcal{F}_{[m+1,m+2]}(x_{m}),\mathcal{F}_{[m+1,m+3]}(x_{m}),\ldots\}.$$
\end{center}

If $x$ is a point which $\varepsilon$-shadows the above $\delta$-pseudo orbit, then
\begin{center}
$d(\mathcal{F}_{[m+1,m+j]}(\mathcal{F}_{m}(x)),\mathcal{F}_{[m+1,m+j]}(x_{m}))<\varepsilon<\gamma$
\end{center}
for all $j\geq 0$ which implies that $\mathcal{F}_{m}(x)=x_{m}$. Thus $\mathcal{F}$ has the h-shadowing property. Finally, (b) is a direct consequence of part (a) and Theorem \ref{theorem22}, which completes the proof.
\end{proof}
\section{Uniformly contracting and uniformly expanding time varying maps}\label{section4}
In this section, we investigate various notions of shadowing for uniformly contracting and uniformly expanding time varying maps. We show that the uniformly contracting and uniformly expanding time varying maps exhibit the shadowing, limit shadowing, s-limit shadowing and exponential limit shadowing properties. Moreover, we show that any time varying map of a finite set of hyperbolic linear homeomorphisms on a Banach space with the same stable and unstable subspaces has the shadowing, limit shadowing, s-limit shadowing and exponential limit shadowing properties.
\begin{definition}[Uniformly contracting and uniformly expanding time varying map]
Let $\mathcal{F}=\{f_{n}\}_{n\in\mathbb{N}}$ be a time varying map on a metric space $(X,d)$. Then,
\begin{enumerate}
\item
the time varying map $\mathcal{F}$ is \emph{uniformly contracting} if its contracting ratio which denoted by $\alpha$ exists and is less than one, where
\begin{equation*}
\alpha:=\sup_{n\in\mathbb{N}}\sup_{\substack{x,y\in X \\
x\neq y}}\dfrac{d(f_{n}(x),f_{n}(y))}{d(x,y)};
\end{equation*}
\item
the time varying map $\mathcal{F}$ is \emph{uniformly expanding} if its expanding ratio which denoted by $\beta$
exists and is greater than one, where
\begin{equation*}
\beta:=\inf_{n\in\mathbb{N}}\inf_{\substack{x,y\in X \\
x\neq y}}\dfrac{d(f_{n}(x),f_{n}(y))}{d(x,y)}.
\end{equation*}
\end{enumerate}
\end{definition}
In the following theorem, we show that uniformly contracting time varying maps exhibit the shadowing, limit shadowing, s-limit shadowing and exponential limit shadowing properties.
\begin{theorem}\label{contracting}
Let $\mathcal{F}=\{f_{n}\}_{n\in\mathbb{N}}$ be a uniformly contracting time varying map on a metric space $(X,d)$. Then,
\begin{itemize}
\item[\textnormal{(a)}]
$\mathcal{F}$ has the shadowing property;
\item[\textnormal{(b)}]
$\mathcal{F}$ has the limit shadowing property;
\item[\textnormal{(c)}]
$\mathcal{F}$ has the exponential limit shadowing property;
\item[\textnormal{(d)}]
$\mathcal{F}$ has the s-limit shadowing property.
\end{itemize}
\end{theorem}
\begin{proof}
Assume that the time varying map $\mathcal{F}$ is uniformly contracting with the contracting ratio $\alpha$.

$\textnormal{(a)}$. Given $\varepsilon>0$ take $\delta=(1-\alpha)\frac{\varepsilon}{2}\leq\frac{\varepsilon}{2}$, and let $\{x_{n}\}_{n\geq 0}$ be a $\delta$-pseudo orbit of $\mathcal{F}$, i.e. $d(f_{n+1}(x_{n}),x_{n+1})<\delta$
for all $n\geq 0$. Consider a point $x\in X$ with $d(x,x_{0})\leq\frac{\varepsilon}{2}$. We show that the
$\delta$-pseudo orbit $\{x_{n}\}_{n\geq 0}$ is $\varepsilon$-shadowed by $x$, i.e. $d(\mathcal{F}_{n}(x),x_{n})<\varepsilon$ for all $n\geq 0$.
Observe that $d(\mathcal{F}_{0}(x),x_{0})\leq\frac{\varepsilon}{2}$ and
\begin{equation*}
d(\mathcal{F}_{1}(x),x_{1})\leq d(\mathcal{F}_{1}(x),f_{1}(x_{0}))+d(f_{1}(x_{0}),x_{1})\leq\alpha d(x,x_{0})+\delta.
\end{equation*}
Similarly,
\begin{eqnarray*}
d(\mathcal{F}_{2}(x),x_{2})
&\leq & d(\mathcal{F}_{2}(x),f_{2}(x_{1}))+d(f_{2}(x_{1}),x_{2})\\
&=& d(f_{2}(\mathcal{F}_{1}(x)),f_{2}(x_{1}))+d(f_{2}(x_{1}),x_{2})\\
&\leq & \alpha d(\mathcal{F}_{1}(x),x_{1})+\delta\\
&\leq & \alpha^{2} d(x,x_{0})+\delta\alpha+\delta.
\end{eqnarray*}
By induction, for each $n\geq 0$, one can show that
\begin{equation*}
d(\mathcal{F}_{n}(x),x_{n})\leq\alpha^{n} d(x,x_{0})+\delta(\alpha^{n-1}+\alpha^{n-2}+\cdots+\alpha+1).
\end{equation*}
Now, the last inequality together with $d(x,x_{0})\leq\frac{\varepsilon}{2}$ imply
\begin{equation*}
d(\mathcal{F}_{n}(x),x_{n})\leq d(x,x_{0})+\delta(\dfrac{1}{1-\alpha})\leq\dfrac{\varepsilon}{2}+\dfrac{\varepsilon}{2}=\varepsilon.
\end{equation*}
Hence, the $\delta$-pseudo orbit $\{x_{n}\}_{n\geq 0}$ is $\varepsilon$-shadowed by $x$ and so time varying map $\mathcal{F}$ has the shadowing property, which completes the proof of part $\textnormal{(a)}$.

$\textnormal{(b)}$. Let $\{x_{n}\}_{n\geq 0}$ be a limit pseudo orbit of $\mathcal{F}$, i.e. $d(f_{n+1}(x_{n}),x_{n+1})\to 0$ as $n\to +\infty$. Put $\tau_{n}=d(f_{n+1}(x_{n}),x_{n+1})$ for all $n\geq 0$ (note that $\tau_{n}\to 0$ as $n\to +\infty$). Now, we show that $d(\mathcal{F}_{n}(x_{0}),x_{n})\to 0$ as $n\to +\infty$, which implies $\{x_{n}\}_{n\geq 0}$ is limit shadowed by $x_{0}$.
To prove this, suppose $\varepsilon$ is an arbitrary positive real number and $M=\sup_{n\geq 0}\tau_{n}$. We can find $k\in\mathbb{N}$ such
that $M\dfrac{\alpha^{k}}{1-\alpha}<\dfrac{\varepsilon}{2}$ and $\tau_{i}<\varepsilon\dfrac{1-\alpha}{2}$ for all $i\geq k$. Obviously, $d(\mathcal{F}_{0}(x_{0}),x_{0})=0$ and
\begin{equation*}
d(\mathcal{F}_{1}(x_{0}),x_{1})\leq d(\mathcal{F}_{1}(x_{0}),f_{1}(x_{0}))+d(f_{1}(x_{0}),x_{1})=d(f_{1}(x_{0}),f_{1}(x_{0}))+\tau_{0}=\tau_{0}.
\end{equation*}
Similarly,
\begin{eqnarray*}
d(\mathcal{F}_{2}(x_{0}),x_{2})
&\leq & d(\mathcal{F}_{2}(x_{0}),f_{2}(x_{1}))+d(f_{2}(x_{1}),x_{2})\\
&=& d(f_{2}(\mathcal{F}_{1}(x_{0})),f_{2}(x_{1}))+d(f_{2}(x_{1}),x_{2})\\
&\leq & \alpha d(\mathcal{F}_{1}(x_{0}),x_{1})+\tau_{1}\\
&\leq & \alpha\tau_{0}+\tau_{1},
\end{eqnarray*}
and
\begin{eqnarray*}
d(\mathcal{F}_{3}(x_{0}),x_{3})
&\leq & d(\mathcal{F}_{3}(x_{0}),f_{3}(x_{2}))+d(f_{3}(x_{2}),x_{3})\\
&=& d(f_{3}(\mathcal{F}_{2}(x_{0})),f_{3}(x_{2}))+d(f_{3}(x_{2}),x_{3})\\
&\leq & \alpha d(\mathcal{F}_{2}(x_{0}),x_{2})+\tau_{2}\\
&\leq & \alpha(\alpha\tau_{0}+\tau_{1})+\tau_{2}\\
&=& \alpha^{2}\tau_{0}+\alpha\tau_{1}+\tau_{2}.
\end{eqnarray*}
By induction, for each $n\geq 0$ we have that
\begin{equation*}
d(\mathcal{F}_{n}(x_{0}),x_{n})\leq\alpha^{n-1}\tau_{0}+\alpha^{n-2}\tau_{1}+\cdots+\alpha\tau_{n-2}+\tau_{n-1}.
\end{equation*}
Hence for $n\geq k$, we have
\begin{eqnarray*}
d(\mathcal{F}_{n}(x_{0}),x_{n})
&\leq & \alpha^{n-1}\tau_{0}+\alpha^{n-2}\tau_{1}+\cdots+\alpha^{k+1}\tau_{n-(k+2)}+\alpha^{k}\tau_{n-(k+1)}\\
&& +\alpha^{k-1}\tau_{n-k}+\cdots+\alpha\tau_{n-2}+\tau_{n-1}\\
&\leq & M\alpha^{k}(1+\alpha+\cdots+\alpha^{n-k-1})+\varepsilon\dfrac{1-\alpha}{2}(1+\alpha+\cdots+\alpha^{k-1})\\
&=& M\alpha^{k}\dfrac{\alpha^{n-k}}{1-\alpha}+\varepsilon\dfrac{1-\alpha}{2}\dfrac{\alpha^{k}}{1-\alpha}\\
&\leq & \dfrac{\varepsilon}{2}+\dfrac{\varepsilon}{2}=\varepsilon.
\end{eqnarray*}
Therefore $d(\mathcal{F}_{n}(x_{0}),x_{n})\leq\varepsilon$ as $n\to +\infty$. Since $\varepsilon>0$ was arbitrary, we conclude that $d(\mathcal{F}_{n}(x_{0}),x_{n})\to 0$ as $n\to +\infty$. Hence, time varying map $\mathcal{F}$ has the limit shadowing property, which completes the proof of part $\textnormal{(b)}$.

$\textnormal{(c)}$. We choose $\theta_{0}\in(\alpha,1)$ and show that $\mathcal{F}$ has the exponential limit shadowing property with respect to this $\theta_{0}$. Let $\{x_{n}\}_{n\geq 0}$ be a $\theta$-exponentially limit pseudo orbit of $\mathcal{F}$ with $\theta\in(\theta_{0},1)$, i.e.
$d(f_{n+1}(x_{n}),x_{n+1})\xrightarrow{\theta} 0$ as $ n\to +\infty$. Then, there exists $L>0$ such that
$d(f_{n+1}(x_{n}),x_{n+1})\leq L\theta^{n}$ for all $n\geq 0$. Hence,

\begin{eqnarray*}
d(\mathcal{F}_{n}(x_{0}),x_{n})
&\leq & d(\mathcal{F}_{n}(x_{0}),f_{n}(x_{n-1}))+d(f_{n}(x_{n-1}),x_{n})\\
&=& d(f_{n}(\mathcal{F}_{n-1}(x_{0})),f_{n}(x_{n-1}))+d(f_{n}(x_{n-1}),x_{n})\\
&\leq & \alpha d(\mathcal{F}_{n-1}(x_{0}),x_{n-1})+L\theta^{n}\\
&\leq & \alpha^{2}d(\mathcal{F}_{n-2}(x_{0}),x_{n-2})+\alpha L\theta^{n-1}+L\theta^{n}\\
&\vdots &\\
&\leq & \alpha^{n-1}L\theta+\alpha^{n-2}L\theta^{2}+\cdots+\alpha^{2}L\theta^{n-2}+\alpha L\theta^{n-1}+L\theta^{n}\\
&= & L(1+\alpha\theta^{-1}+\alpha^{2}\theta^{-2}+\cdots+\alpha^{n-1}\theta^{-n+1})\theta^{n}\\
&\leq & L(1+\alpha\theta^{-1}+\alpha^{2}\theta^{-2}+\cdots+\alpha^{n-1}\theta^{-n+1})\theta^{n-1}\\
&\leq & (\dfrac{L}{\theta-\alpha})\theta^{n}.
\end{eqnarray*}
Therefore $d(\mathcal{F}_{n}(x_{0}),x_{n})\xrightarrow{\theta} 0$, as $n\to +\infty$, i.e. time varying map $\mathcal{F}$ has the exponential limit shadowing property, which completes the proof of part $\textnormal{(c)}$.

Finally, the part $\textnormal{(d)}$ is a direct consequence of our process in parts $\textnormal{(a)}$ and $\textnormal{(b)}$, which completes the proof of the theorem.
\end{proof}
\begin{remark}
In general a time varying map with shadowing and limit shadowing properties does not have s-limit shadowing property, see example \ref{example1}.
\end{remark}
\begin{corollary}
Let $I$ be a non-empty finite set and for every $i\in I$, $f_{i}:\mathbb{R}\to\mathbb{R}$ be a differentiable function.  Assume that the maps $f_{i}$ have a common attractor fixed point $p\in\mathbb{R}$, i.e. $f_{i}(p)=p$
and $|f_{i}^{\prime}(p)|<1$ for all $i\in I$. Set $\mathcal{A}=\{f_{i}\}_{i\in I}$. Then, there is an open interval $U$ about $p$ such that $f(U)\subset U$, for all $f\in\mathcal{A}$. Moreover, each time varying map $\mathcal{F}=\{f_{n}\}_{n\in\mathbb{N}}$ with $f_{n}\in\mathcal{A}$ on $U$ has the shadowing, limit shadowing, s-limit shadowing and exponential limit shadowing properties.
\end{corollary}
\begin{proof}
By \cite[Proposition 4.4]{RDEV}, for each $i\in I$ there is an open interval $U_{i}$ about $p$ such that if $x\in U_{i}$, then $f_{i}(x)\in U_{i}$ and $f_{i}^{n}(x)\to p$ as $n\to +\infty$. Hence, we can find an open interval $U\subset\cap_{i\in I}U_{i}$ and $0<\varepsilon<1$ such that if $x\in U$, then $|f_{i}^{\prime}(x)|<1-\varepsilon$, $f_{i}(x)\in U$ and $f_{i}^{n}(x)\to p$ as $n\to +\infty$, for every $i\in I$. This implies that for all $x,y\in U$, we have
$\dfrac{|f_{i}(x)-f_{i}(y)|}{|x-y|}<1-\varepsilon$. Hence, each time varying map $\mathcal{F}=\{f_{n}\}_{n\in\mathbb{N}}$ with $f_{n}\in\mathcal{A}$ on $U$ is uniformly contracting, and so by Theorem \ref{contracting} it has the shadowing, limit shadowing, s-limit shadowing and exponential limit shadowing properties.
\end{proof}
In the following, we provide two uniformly contracting time varying maps.
\begin{example}
Let $\Sigma_{2}:=\{0,1\}^{\mathbb{N}}=\{x=(x_{1}x_{2}\ldots): x_{n}\in\{0,1\}\}$ be the Bernoulli space.
Consider in $\Sigma_{2}$ the distance defined by
\begin{equation*}
d(x,y) = \left\{
\begin{array}{rl}
2^{-N} & \text{if } x\neq y\ \textnormal{and}\ N=\min\{i:x_{i}\neq y_{i}\},\\
0 & \text{if } x=y.
\end{array} \right.
\end{equation*}
Now, let $f,g:\Sigma_{2}\to\Sigma_{2}$ be two maps defined as
follows
\begin{equation*}
f((x_{1}x_{2}\ldots))=(0x_{1}x_{2}\ldots)\qquad g((x_{1}x_{2}\ldots))=(1x_{1}x_{2}\ldots).
\end{equation*}
Then, any time varying map $\mathcal{F}=\{f_{n}\}_{n\in\mathbb{N}}$ with $f_{n}\in\{f,g\}$ is uniformly contracting, and so by Theorem \ref{contracting} it has the shadowing, limit shadowing, s-limit shadowing and exponential limit shadowing properties.
\end{example}
\begin{example}
Let $X\subset \mathbb{R}^2$ be the Sierpinski triangle on the solid equilateral triangle
which is constructed by repeatedly removing inverted maximal equilateral (solid) triangles from a given equilateral (solid)
triangle. Denote the sets in this construction by $X_0, X_1, \cdots$, whereby $X=\bigcap_{n=0}^{\infty}X_n$.
Then $X$ is selfsimilar,
\begin{equation*}
  X=\bigcup_{i=1}^3 g_i(X),
\end{equation*}
where the $g_1, g_2, g_3:\mathbb{R}^2 \to \mathbb{R}^2$ are the homotheties of rate $1/2$ that keep one of the three
vertices of $X_0$ fixed, see \cite{Ri} for more details.
Then $g_i$, $i=1, 2, 3$, are uniformly contracting. Hence, by Theorem \ref{contracting}, each time varying map $\mathcal{F}=\{f_{n}\}_{n\in\mathbb{N}}$ with $f_{n}\in \{g_1, g_2, g_3\}$ is uniformly contracting, and so it has the shadowing, limit shadowing, s-limit shadowing and exponential limit shadowing properties.
\end{example}
In the following theorem, we show that every uniformly expanding time varying map has the shadowing, limit shadowing, s-limit shadowing and exponential limit shadowing properties. Note that Castro, Rodrigues and Varandas (\cite[Lemma 2.1]{CARFBVP}) assert that the uniformly expanding time varying maps having the shadowing property. Also, Nazarian Sarkooh and Ghane (\cite[Proposition 4.12]{JNSFHG}) showed that the uniformly Ruelle-expanding time varying maps having the shadowing property. Here, we give a different proof for shadowing property of uniformly expanding time varying maps and use it to yield the limit shadowing, s-limit shadowing and exponential limit shadowing properties. Recall that an expanding and surjective map is invertible.
\begin{theorem}\label{expanding}
Let $\mathcal{F}=\{f_{n}\}_{n\in\mathbb{N}}$ be a uniformly expanding time varying map of surjective maps on a complete metric space $(X,d)$. Then,
\begin{itemize}
\item[\textnormal{(a)}]
$\mathcal{F}$ has the shadowing property;
\item[\textnormal{(b)}]
$\mathcal{F}$ has the limit shadowing property;
\item[\textnormal{(c)}]
$\mathcal{F}$ has the exponential limit shadowing property;
\item[\textnormal{(d)}]
$\mathcal{F}$ has the s-limit shadowing property.
\end{itemize}
\end{theorem}
\begin{proof}
Here, we use the approach used in the proof of \cite[Theorem 2.2]{VGVG}. Given $n\in\mathbb{N}$, consider the function $\varphi_{n}:X\times X\to[0,+\infty)$, defined by
\begin{equation*}
\varphi_{n}(x,y) = \left\{
\begin{array}{rl}
\dfrac{d(f_{n}(x),f_{n}(y))}{d(x,y)} & \text{if } x\neq y,\\
\beta & \text{if } x=y,
\end{array} \right.
\end{equation*}
where $\beta>1$ is the expanding ratio of the uniformly expanding time varying map $\mathcal{F}$. Hence, one has
\begin{equation}\label{jjj1}
d(x,y)=\dfrac{d(f_{n}(x),f_{n}(y))}{\varphi_{n}(x,y)},\quad \varphi_{n}(x,y)\geq\beta\quad\textnormal{for all}\ x,y\in X\ \textnormal{and}\ n\in\mathbb{N}.
\end{equation}

$\textnormal{(a)}$. Given $\varepsilon>0$ take $\delta=(\beta-1)\varepsilon$, and let $\{x_{n}\}_{n\geq 0}$ be a $\delta$-pseudo orbit of $\mathcal{F}$, i.e. $d(f_{n+1}(x_{n}),x_{n+1})<\delta$ for all $n\geq 0$. Consider
the sequence $\{z_{n}\}_{n\geq 0}$ in $X$, defined as follows
\begin{equation}\label{jjj7}
z_{0}=x_{0},\quad z_{n}=f_{1}^{-1}\circ f_{2}^{-1}\circ\cdots\circ f_{n}^{-1}(x_{n})\quad\textnormal{for all}\ n\in\mathbb{N}.
\end{equation}
Then, $x_{n}=f_{n}\circ f_{n-1}\circ\cdots\circ f_{1}(z_{n})$ for all $n\in\mathbb{N}$. Given $n\in\mathbb{N}$ and $1\leq k\leq n$, denote
\begin{equation}\label{jjj2}
z_{n}^{(k)}=f_{k}\circ f_{k-1}\circ\cdots\circ f_{1}(z_{n}).
\end{equation}
Therefore, for any $n\in\mathbb{N}$ and $1\leq k\leq n$, one has
\begin{equation}\label{jjj3}
z_{n}^{(k)}=f_{k}(z_{n}^{(k-1)})\qquad x_{n}=z_{n}^{(n)}=f_{n}(z_{n}^{(n-1)}).
\end{equation}
We claim that $\{z_{n}\}_{n\geq 0}$ is a Cauchy sequence. Firstly, fixing $n\in\mathbb{N}$ and
$p\geq 1$, and using (\ref{jjj1}), (\ref{jjj2}) and (\ref{jjj3}), we obtain
\begin{eqnarray}\label{jjj4}
d(z_{n},z_{n+p})
&=& \dfrac{d(f_{1}(z_{n}),f_{1}(z_{n+p}))}{\varphi_{1}(z_{n},z_{n+p})}\notag\\
&=& \dfrac{d(z_{n}^{(1)},z_{n+p}^{(1)})}{\varphi_{1}(z_{n},z_{n+p})}\notag\\
&=& \dfrac{d(z_{n}^{(2)},z_{n+p}^{(2)})}{\varphi_{1}(z_{n},z_{n+p})\varphi_{2}(z_{n}^{(1)},z_{n+p}^{(1)})}\notag\\
&\vdots &\notag\\
&=& \dfrac{d(z_{n}^{(n)},z_{n+p}^{(n)})}{\varphi_{1}(z_{n},z_{n+p})\prod_{i=2}^{n}\varphi_{i}(z_{n}^{(i-1)},z_{n+p}^{(i-1)})}\notag\\
&=& \dfrac{d(x_{n},z_{n+p}^{(n)})}{\varphi_{1}(z_{n},z_{n+p})\prod_{i=2}^{n}\varphi_{i}(z_{n}^{(i-1)},z_{n+p}^{(i-1)})}.
\end{eqnarray}
Secondly, by induction on $p\geq 1$ we show that the following inequality holds
uniformly with respect to $n\in\mathbb{N}$:
\begin{equation}\label{jjj5}
d(x_{n},z_{n+p}^{(n)})\leq\delta\sum_{k=1}^{p}\beta^{-k}.
\end{equation}
Indeed, for $p=1$ the inequality (\ref{jjj5}) follows from (\ref{jjj1}) and (\ref{jjj3}):
\begin{equation}
d(x_{n},z_{n+1}^{(n)})=\dfrac{d(f_{n+1}(x_{n}),f_{n+1}(z_{n+1}^{(n)})}{\varphi_{n+1}(x_{n},z_{n+1}^{(n)})}=
\dfrac{d(f_{n+1}(x_{n}),x_{n+1})}{\varphi_{n+1}(x_{n},z_{n+1}^{(n)})}\leq\dfrac{\delta}{\beta}.
\end{equation}
Assume that (\ref{jjj5}) holds for some $p=q\geq 1$ uniformly on $n\in\mathbb{N}$. Taking into account this assumption, as well as (\ref{jjj1}) and (\ref{jjj3}), we prove (\ref{jjj5}) for
$p=q+1$:
\begin{eqnarray}\label{jjj9}
d(x_{n},z_{n+q+1}^{(n)})
&=& \dfrac{d(f_{n+1}(x_{n}),f_{n+1}(z_{n+q+1}^{(n)})}{\varphi_{n+1}(x_{n},z_{n+q+1}^{(n)})}\notag\\
&=& \dfrac{d(f_{n+1}(x_{n}),z_{n+1+q}^{(n+1)})}{\varphi_{n+1}(x_{n},z_{n+q+1}^{(n)})}\notag\\
&\leq & \dfrac{d(f_{n+1}(x_{n}),x_{n+1})+d(x_{n+1},z_{n+1+q}^{(n+1)})}{\varphi_{n+1}(x_{n},z_{n+1+q}^{(n)})}\notag\\
&\leq & \dfrac{1}{\beta}(\delta+\delta\sum_{k=1}^{q}\beta^{-k})=\delta\sum_{k=1}^{q+1}\beta^{-k}.
\end{eqnarray}
Therefore (\ref{jjj5}) holds.

Now, the relations (\ref{jjj1}), (\ref{jjj4}) and (\ref{jjj5}) give us the following estimation for $d(z_{n},z_{n+p})$ with any $n\in\mathbb{N}$ and $p\geq 1$:
\begin{eqnarray}\label{jjj6}
d(z_{n},z_{n+p})
&\leq & \dfrac{\delta\sum_{k=1}^{p}\beta^{-k}}{\varphi_{1}(z_{n},z_{n+p})\prod_{i=2}^{n}\varphi_{i}(z_{n}^{(i-1)},z_{n+p}^{(i-1)})}\notag\\
&\leq & \dfrac{\delta}{(\beta-1)\varphi_{1}(z_{n},z_{n+p})\prod_{i=2}^{n}\varphi_{i}(z_{n}^{(i-1)},z_{n+p}^{(i-1)})}\notag\\
&=& \dfrac{\varepsilon}{\varphi_{1}(z_{n},z_{n+p})\prod_{i=2}^{n}\varphi_{i}(z_{n}^{(i-1)},z_{n+p}^{(i-1)})}\leq\varepsilon\beta^{-n}.
\end{eqnarray}
This inequality proves the claim, i.e. $\{z_{n}\}_{n\geq 0}$ is a Cauchy sequence.
Therefore, the sequence $\{z_{n}\}_{n\geq 0}$ is convergent to some point $x\in X$. From (\ref{jjj2}) and (\ref{jjj6}) one has
\begin{equation*}
\lim_{n\to\infty} z_{n}^{(k)}=f_{k}\circ f_{k-1}\circ\cdots\circ f_{1}(x)=\mathcal{F}_{k}(x)\ \textnormal{for any}\ k\geq 1,
\end{equation*}
and
\begin{equation*}
d(z_{n},x)\leq\dfrac{\varepsilon}{\varphi_{1}(z_{n},x)\prod_{i=2}^{n}\varphi_{i}(z_{n}^{(i-1)},\mathcal{F}_{i-1}(x))}\quad\textnormal{as}\ p\to\infty,\ \textnormal{for}\ n\geq 1.
\end{equation*}
Hence, for $n\geq 1$ we get
\begin{eqnarray*}
d(\mathcal{F}_{n}(x),x_{n})
&=& d(f_{n}\circ\cdots\circ f_{1}(x),f_{n}\circ\cdots\circ f_{1}(z_{n}))\\
&=& \varphi_{n}(\mathcal{F}_{n-1}(x),z_{n}^{n-1})d(f_{n-1}\circ\cdots\circ f_{1}(x),f_{n-1}\circ\cdots\circ f_{1}(z_{n}))\\
&=& \varphi_{n}(\mathcal{F}_{n-1}(x),z_{n}^{n-1})\varphi_{n-1}(\mathcal{F}_{n-2}(x),z_{n}^{n-2})d(f_{n-2}\circ\cdots\circ f_{1}(x),f_{n-2}\circ\cdots\circ f_{1}(z_{n}))\\
&\vdots&\\
&=& \varphi_{n}(\mathcal{F}_{n-1}(x),z_{n}^{n-1})\varphi_{n-1}(\mathcal{F}_{n-2}(x),z_{n}^{n-2})\cdots\varphi_{1}(x,z_{n})
d(x,z_{n})\leq\varepsilon.
\end{eqnarray*}
Also, for the lacking case $n=0$
\begin{eqnarray*}
d(\mathcal{F}_{0}(x),x_{0})
&=& d(x,x_{0})=\dfrac{d(f_{1}(x),f_{1}(x_{0}))}{\varphi_{1}(x,x_{0})}\\
&=& \dfrac{d(f_{1}(x),x_{1})+d(x_{1},f_{1}(x_{0}))}{\varphi_{1}(x,x_{0})}\\
&=& \dfrac{d(\mathcal{F}_{1}(x),x_{1})+d(x_{1},f_{1}(x_{0}))}{\varphi_{1}(x,x_{0})}\\
&\leq & \dfrac{\varepsilon+\delta}{\beta}=\dfrac{\varepsilon+\beta\varepsilon-\varepsilon}{\beta}=\varepsilon.
\end{eqnarray*}
Hence, $\{x_{n}\}_{n\geq 0}$ is $\varepsilon$-shadowed by $x$. Thus $\mathcal{F}$ has the shadowing property which completes the proof of part $\textnormal{(a)}$.

$\textnormal{(b)}$. Let $\varepsilon>0$ and $\{x_{n}\}_{n\geq 0}$ be a limit pseudo orbit of $\mathcal{F}$, i.e. $d(f_{n+1}(x_{n}),x_{n+1})\to 0$ as $ n\to +\infty$. Therefore, there is $N_{0}\in\mathbb{N}$ such that $d(f_{n+1}(x_{n}),x_{n+1})<\varepsilon$, for all $n\geq N_{0}$. Now, consider the sequence $\{z_{n}\}_{n\geq 0}$ and notation $z_{n}^{(k)}$ as given by relations (\ref{jjj7}) and (\ref{jjj2}). We claim that $\{z_{n}\}_{n\geq 0}$ is a Cauchy sequence. Hence, by induction on $p\geq 1$ we show that the following inequality holds
uniformly with respect to $n\geq N_{0}$:
\begin{equation}\label{jjj11}
d(x_{n},z_{n+p}^{(n)})\leq\varepsilon\sum_{k=1}^{p}\beta^{-k}.
\end{equation}
Indeed, for $p=1$ the inequality (\ref{jjj11}) follows from (\ref{jjj1}) and (\ref{jjj3}):
\begin{equation*}
d(x_{n},z_{n+1}^{(n)})=\dfrac{d(f_{n+1}(x_{n}),f_{n+1}(z_{n+1}^{(n)})}{\varphi_{n+1}(x_{n},z_{n+1}^{(n)})}=
\dfrac{d(f_{n+1}(x_{n}),x_{n+1})}{\varphi_{n+1}(x_{n},z_{n+1}^{(n)})}\leq\dfrac{\varepsilon}{\beta}.
\end{equation*}
Assume that (\ref{jjj11}) holds for some $p=q\geq 1$ uniformly on $n\geq N_{0}$. Taking into account this assumption, as well as (\ref{jjj1}), (\ref{jjj3}) and (\ref{jjj9}), we prove (\ref{jjj11}) for
$p=q+1$:
\begin{eqnarray*}
d(x_{n},z_{n+q+1}^{(n)})
&\leq & \dfrac{d(f_{n+1}(x_{n}),x_{n+1})+d(x_{n+1},z_{n+1+q}^{(n+1)})}{\varphi_{n+1}(x_{n},z_{n+1+q}^{(n)})}\\
&\leq & \dfrac{1}{\beta}(\varepsilon+\varepsilon\sum_{k=1}^{q}\beta^{k})
=\varepsilon\sum_{k=1}^{q+1}\beta^{-k}.
\end{eqnarray*}
Therefore (\ref{jjj11}) holds. Now, the relations (\ref{jjj4}) and (\ref{jjj11}) give us the following estimation for $d(z_{n},z_{n+p})$ with any $n\geq N_{0}$ and $p\geq 1$:
\begin{eqnarray}\label{jjj12}
d(z_{n},z_{n+p})
&\leq & \dfrac{\varepsilon\sum_{k=1}^{p}\beta^{-k}}{\varphi_{1}(z_{n},z_{n+p})\prod_{i=2}^{n}\varphi_{i}(z_{n}^{(i-1)},z_{n+p}^{(i-1)})}\notag\\
&\leq & \dfrac{\varepsilon}{(\beta-1)\varphi_{1}(z_{n},z_{n+p})\prod_{i=2}^{n}\varphi_{i}(z_{n}^{(i-1)},z_{n+p}^{(i-1)})}\notag\\
&\leq & \dfrac{\varepsilon}{(\beta-1)}\beta^{-n}.
\end{eqnarray}
This inequality proves the claim, i.e. $\{z_{n}\}_{n\geq 0}$ is a Cauchy sequence.
Therefore, the sequence $\{z_{n}\}_{n\geq 0}$ is convergent to some point $x\in X$. Also, from (\ref{jjj12}) one has
\begin{equation*}
d(z_{n},x)\leq\dfrac{\varepsilon}{(\beta-1)\varphi_{1}(z_{n},x)\prod_{i=2}^{n}\varphi_{i}(z_{n}^{(i-1)},\mathcal{F}_{i-1}(x))}\quad\textnormal{as}\ p\to\infty,\ \textnormal{for}\ n\geq 1.
\end{equation*}
Hence, for $n\geq N_{0}$ we get
\begin{eqnarray*}
d(\mathcal{F}_{n}(x),x_{n})
&=& d(f_{n}\circ\cdots\circ f_{1}(x),f_{n}\circ\cdots\circ f_{1}(z_{n}))\\
&=& \varphi_{n}(\mathcal{F}_{n-1}(x),z_{n}^{n-1})d(f_{n-1}\circ\cdots\circ f_{1}(x),f_{n-1}\circ\cdots\circ f_{1}(z_{n}))\\
&\vdots&\\
&=& \varphi_{n}(\mathcal{F}_{n-1}(x),z_{n}^{n-1})\varphi_{n-1}(\mathcal{F}_{n-2}(x),z_{n}^{n-2})\cdots\varphi_{1}(x,z_{n})
d(x,z_{n})\leq\dfrac{\varepsilon}{\beta-1}
\end{eqnarray*}
that implies $d(\mathcal{F}_{n}(x),x_{n})\to 0$ as $n\to\infty$, because $\varepsilon$ is arbitrary. Thus the time varying map $\mathcal{F}$ has the limit shadowing property.

$\textnormal{(c)}$. Let $\{x_{n}\}_{n\geq 0}$ be a $\theta$-exponentially limit pseudo orbit of $\mathcal{F}$ with rate $\theta\in(0,1)$, i.e. $d(f_{n+1}(x_{n}),x_{n+1})\xrightarrow{\theta} 0$ as $ n\to +\infty$. Hence, there exists a constant $L>0$ such that $d(f_{n+1}(x_{n}),x_{n+1})\leq L\theta^{n}$ for $n\geq 0$.
Consider the sequence $\{z_{n}\}_{n\geq 0}$ and notation $z_{n}^{(k)}$ given by relations (\ref{jjj7}) and (\ref{jjj2}). We claim that $\{z_{n}\}_{n\geq 0}$ is a Cauchy sequence. Hence, by induction on $p\geq 1$ we show that the following inequality holds
uniformly with respect to $n\in\mathbb{N}$:
\begin{equation}\label{jjj8}
d(x_{n},z_{n+p}^{(n)})\leq L\delta^{n}\sum_{k=1}^{p}\beta^{-k}.
\end{equation}
Indeed, for $p=1$ the inequality (\ref{jjj8}) follows from (\ref{jjj1}) and (\ref{jjj3}):
\begin{equation*}
d(x_{n},z_{n+1}^{(n)})=\dfrac{d(f_{n+1}(x_{n}),f_{n+1}(z_{n+1}^{(n)})}{\varphi_{n+1}(x_{n},z_{n+1}^{(n)})}=
\dfrac{d(f_{n+1}(x_{n}),x_{n+1})}{\varphi_{n+1}(x_{n},z_{n+1}^{(n)})}\leq\dfrac{L\delta^{n}}{\beta}.
\end{equation*}
Assume that (\ref{jjj8}) holds for some $p=q\geq 1$ uniformly on $n\in\mathbb{N}$. Taking into account this assumption, as well as (\ref{jjj1}), (\ref{jjj3}) and (\ref{jjj9}), we prove (\ref{jjj8}) for
$p=q+1$:
\begin{eqnarray*}
d(x_{n},z_{n+q+1}^{(n)})
&\leq & \dfrac{d(f_{n+1}(x_{n}),x_{n+1})+d(x_{n+1},z_{n+1+q}^{(n+1)})}{\varphi_{n+1}(x_{n},z_{n+1+q}^{(n)})}\\
&\leq & \dfrac{1}{\beta}(L\theta^{n}+L\theta^{n+1}\sum_{k=1}^{q}\beta^{-k})\\
&\leq & \dfrac{1}{\beta}(L\theta^{n}+L\theta^{n}\sum_{k=1}^{q}\beta^{-k})=L\theta^{n}\sum_{k=1}^{q+1}\beta^{-k}.
\end{eqnarray*}
Therefore (\ref{jjj8}) holds. Now, the relations (\ref{jjj4}) and (\ref{jjj8}) give us the following estimation for $d(z_{n},z_{n+p})$ with any $n\in\mathbb{N}$ and $p\geq 1$:
\begin{eqnarray}\label{jjj10}
d(z_{n},z_{n+p})
&\leq & \dfrac{L\theta^{n}\sum_{k=1}^{p}\beta^{-k}}{\varphi_{1}(z_{n},z_{n+p})\prod_{i=2}^{n}\varphi_{i}(z_{n}^{(i-1)},z_{n+p}^{(i-1)})}\notag\\
&\leq & \dfrac{L\theta^{n}}{(\beta-1)\varphi_{1}(z_{n},z_{n+p})\prod_{i=2}^{n}\varphi_{i}(z_{n}^{(i-1)},z_{n+p}^{(i-1)})}\notag\\
&\leq & \dfrac{L}{(\beta-1)}\Big(\dfrac{\theta}{\beta}\Big)^{n}.
\end{eqnarray}
This inequality proves the claim, i.e. $\{z_{n}\}_{n\geq 0}$ is a Cauchy sequence.
Therefore, the sequence $\{z_{n}\}_{n\geq 0}$ is convergent to some point $x\in X$. Also, from (\ref{jjj10}) one has
\begin{equation*}
d(z_{n},x)\leq\dfrac{L\theta^{n}}{(\beta-1)\varphi_{1}(z_{n},x)\prod_{i=2}^{n}\varphi_{i}(z_{n}^{(i-1)},\mathcal{F}_{i-1}(x))}\quad\textnormal{as}\ p\to\infty,\ \textnormal{for}\ n\geq 1.
\end{equation*}
Hence, for $n\geq 1$ we get
\begin{eqnarray*}
d(\mathcal{F}_{n}(x),x_{n})
&=& d(f_{n}\circ\cdots\circ f_{1}(x),f_{n}\circ\cdots\circ f_{1}(z_{n}))\\
&=& \varphi_{n}(\mathcal{F}_{n-1}(x),z_{n}^{n-1})d(f_{n-1}\circ\cdots\circ f_{1}(x),f_{n-1}\circ\cdots\circ f_{1}(z_{n}))\\
&\vdots&\\
&=& \varphi_{n}(\mathcal{F}_{n-1}(x),z_{n}^{n-1})\varphi_{n-1}(\mathcal{F}_{n-2}(x),z_{n}^{n-2})\cdots\varphi_{1}(x,z_{n})
d(x,z_{n})\leq\Big(\dfrac{L}{\beta-1}\Big)\theta^{n}
\end{eqnarray*}
that implies $d(\mathcal{F}_{n}(x),x_{n})\xrightarrow{\theta} 0$ as $ n\to +\infty$. Thus the time varying map $\mathcal{F}$ has the exponential limit shadowing property.

Finally, part $\textnormal{(d)}$ is a direct consequence of our process in parts $\textnormal{(a)}$ and $\textnormal{(b)}$, which completes the proof of the theorem.
\end{proof}
\begin{remark}
Note that the surjectivity of maps $f_{n}$ of time varying map $\mathcal{F}=\{f_{n}\}_{n\in\mathbb{N}}$ in Theorem \ref{expanding} is essential. Indeed, let $X=\{1\}\cup[2,+\infty)$ be a complete metric space endowed with the standard metric from $\mathbb{R}$, and consider the function $f:X\to X$, $f(x)=2x$ for all $x\in X$. Hence $f$ is uniformly expanding that is not surjective. Then time varying map $\mathcal{F}=\{f_{n}\}_{n\in\mathbb{N}}$  on $X$ with $f_{n}=f$ for all $n\in\mathbb{N}$ does not possess the shadowing property, for more details see \cite[Example 2.1]{VGVG}.
\end{remark}
\begin{remark}
Memarbashi and Rasuli in \cite[Theorem 2.8]{MMMMMHHH} show that each uniformly expanding time varying map satisfies the h-shadowing property. Hence, by theorem \ref{theorem22}, each uniformly expanding time varying map on a compact metric space has the s-limit shadowing property, moreover it has the limit shadowing property (under some conditions). Nevertheless, for Theorem \ref{expanding} we give a different proof and use it to yield the exponential limit shadowing property, also it is useful for future studies.
\end{remark}
\begin{example}\label{example2}
Let $f_{A}:\mathbb{T}^{d}\to\mathbb{T}^{d}$ be the linear endomorphism of the torus $\mathbb{T}^{d}=\mathbb{R}^{d}/\mathbb{Z}^{d}$ induced by
some matrix $A$ with integer coefficients and determinant different from zero. Assume that all the
eigenvalues $\lambda_{1},\lambda_{2},\ldots,\lambda_{d}$ of $A$ are larger than $1$ in
absolute value. Then, given any $1<\sigma<\inf_{i}|\lambda_{i}|$, there exists an inner product
in $\mathbb{R}^{d}$ relative to which $||Av||\geq\sigma ||v||$ for every $v\in\mathbb{R}^{d}$. This shows that the transformation $f_{A}$ is expanding, see \cite[Example 6.2]{JNSFHG}.

Now, let $\mathcal{A}$ be a non-empty finite set of different matrices enjoying the above conditions. Then, each time varying map $\mathcal{F}=\{f_{n}\}_{n\in\mathbb{N}}$ with $f_{n}\in\{f_{A}:A\in\mathcal{A}\}$ is uniformly expanding. Hence, by Theorem \ref{expanding}, it has the shadowing, limit shadowing, s-limit shadowing and exponential limit shadowing properties.
\end{example}
In what follows, we show that any time varying map of a finite set of hyperbolic linear homeomorphisms on a Banach space with the same stable and unstable subspaces has the shadowing, limit shadowing, s-limit shadowing and exponential limit shadowing properties.
\begin{definition}
Let $f:X\to X$ be a linear homeomorphism on a Banach space $X$. Then, $f$ is said to be \emph{hyperbolic} if there exist Banach subspaces $X_{s},X_{u}\subset X$, called stable and unstable subspaces, respectively, and a norm $\|.\|$ on $X$ compatible with the original Banach structure such that
\begin{equation*}
X=X_{s}\oplus X_{u},\quad f(X_{s})=X_{s},\quad f(X_{u})=X_{u},\quad \|f|_{X_{s}}\|<1\ \ \textnormal{and}\ \
\|f^{-1}|_{X_{u}}\|<1.
\end{equation*}
\end{definition}
\begin{theorem}\label{hyperbolic}
Let $X$ be a Banach space, and let $\mathcal{A}$ be a finite set of hyperbolic linear homeomorphisms with the same stable and unstable subspaces. Then, any time varying map $\mathcal{F}=\{f_{n}\}_{n\in\mathbb{N}}$ with $f_{n}\in\mathcal{A}$ has the shadowing, limit shadowing, s-limit shadowing and exponential limit shadowing properties.
\end{theorem}
\begin{proof}
Take $\mathcal{G}=\{f_{n}|_{X_{s}}\}_{n\in\mathbb{N}}$ and
$\mathcal{H}=\{f_{n}|_{X_{u}}\}_{n\in\mathbb{N}}$. Then, obviously $\mathcal{G}$ is uniformly
contracting and $\mathcal{H}$ is uniformly expanding on $X_{s}$ and $X_{u}$, respectively. Hence, by Theorems \ref{contracting} and \ref{expanding}, $\mathcal{G}$ and $\mathcal{H}$ have the shadowing, limit shadowing, s-limit shadowing and exponential limit shadowing properties. On the other hand, $\mathcal{F}=\mathcal{G}\times \mathcal{H}$. So, by \cite[Theorem 3.2]{DTRD} and Theorem \ref{theorem1}, time varying map $\mathcal{F}$ has the shadowing, limit shadowing, s-limit shadowing and exponential limit shadowing properties.
\end{proof}
\section*{Acknowledgements}
We would like to thank anonymous reviewer whose remarks improved
the presentation of the paper.

\end{document}